\documentclass[11pt,english,a4paper]{smfart}
\usepackage[english]{babel}
\usepackage{enumerate}
\usepackage{amsmath}
\usepackage{extarrows}
\usepackage{stmaryrd}

\newtheorem{theorem}{Theorem}[section]
\newtheorem{lemma}[theorem]{Lemma}

\newtheorem{corollary}[theorem]{Corollary}
\newtheorem{proposition}[theorem]{Proposition}

\newtheorem*{coro*}{Corollary}

{\theoremstyle{definition}}

{\theoremstyle{definition}}

\theoremstyle{definition}
\newtheorem{definition}[theorem]{Definition}
\newtheorem{question}[theorem]{Question}
\newtheorem{fact}[theorem]{Fact}
\newtheorem{claim}[theorem]{Claim}
\newtheorem*{fact*}{Fact}

\newtheorem*{claim*}{\rm Claim}

{\theoremstyle{definition}\newtheorem{remark}[theorem]{Remark}}

\def\NN{\mathbb N}
\def\ZZ{\mathbb Z}
\def\RR{\mathbb R}

\def\TT{\mathbb T}
\def\PP{\mathbb P}

\def\T{\ensuremath{\mathbb T}}

\def\Z{\ensuremath{\mathbb Z}}

\def\N{\ensuremath{\mathbb N}}
\def\P{\ensuremath{\mathbb P}}

\newcommand{\fhy}{frequently hypercyclic}

\newcommand{\op}{operator}

\newcommand{\inv}{invariant}

\newcommand{\erg}{ergodic}
\newcommand{\mea}{measure}
\newcommand{\prob}{probability}

\newcommand{\To}{\longrightarrow}

\begin{document}

\date{\today}

\title[Invariant measures for frequently hypercyclic operators]{Invariant measures for\\ frequently hypercyclic operators}

\author{Sophie Grivaux}
\address{CNRS,
Laboratoire Paul Painlev\' e, UMR 8524, Universit\'e Lille 1, Cit\' e Scientifique, 59655 Villeneuve d'Ascq
Cedex, France}
\email{grivaux@math.univ-lille1.fr}

\author{\'Etienne Matheron}
\address{Laboratoire de Math\'ematiques de Lens, Universit\'e d'Artois, rue Jean Souvraz S. P. 18, 62307 Lens (France).}
\email{etienne.matheron@euler.univ-artois.fr}

\subjclass{47A16, 47A35, 37A05}

\keywords{Linear dynamical systems, frequently hypercyclic operators, chaotic operators, invariant and ergodic measures}

\thanks{We are grateful to B. Weiss for pointing out to us the example mentioned in Section \ref{basic}.}

\begin{abstract} We investigate frequently hypercyclic and chaotic linear operators from a measure-theoretic point of view. Among other things, we show that any frequently hypercyclic operator $T$ acting on a reflexive Banach space admits an invariant probability measure with full support, which may be required to vanish on the set of all periodic vectors for $T$; that there exist frequently hypercyclic operators on the sequence space $c_0$ admitting no ergodic measure with full support; and that if an operator admits an ergodic measure with full support, then it has a comeager set of distributionally irregular vectors. We also give some necessary and sufficient conditions (which are are satisfied by all the known chaotic operators) for an operator $T$ to admit an invariant measure supported on the set of its hypercyclic vectors and belonging to the closed convex hull of its periodic measures. Finally, we give a Baire category proof of the fact that any operator with a perfectly spanning set of unimodular eigenvectors admits an ergodic measure with full support.
\end{abstract}

\maketitle

\section{Introduction and main results}\label{sec1}

\subsection{General background} Let $X$ be a separable, infinite-dimensional Banach space, and let us denote by $\mathfrak L(X)$ the set of all continuous linear operators on $X$. If $T\in\mathfrak L(X)$,  the pair $(X,T)$ is  called a \emph{linear dynamical system}. 
For every $x\in X$, we denote by $Orb(x,T)$ the orbit of $x$ under the action of $T$, 
$$Orb(x,T)=\{ T^nx;\; n\in\N\}\, .$$

\par\smallskip
A linear dynamical system $(X,T)$ is a special case of a \emph{Polish dynamical system}, \mbox{i.e.} a continuous map acting on a Polish space. As such, it can be studied 
from different points of view. 

- One can adopt a purely topological viewpoint, investigating in particular the individual behaviour of orbits. A basic notion in this context is that of \emph{hypercyclicity}: the 
operator $T$ is said to be hypercyclic if there exists a vector $x\in X$, called a hypercyclic vector for $T$, whose orbit $Orb(x,T)$ is dense in $X$. It is well-known that $T$ is 
hypercyclic if and only if it is topologically transitive, \mbox{i.e.} for each pair $(U,V)$ of nonempty open sets in $X$, one can find $n\in\mathbb N$ such that 
$T^n(U)\cap V\neq\emptyset$; and in this case, the set of hypercyclic vectors for $T$, denoted by $HC(T)$, is a dense $G_\delta$ subset of $X$.

- One can also adopt a measure-theoretic point of view and try to find conditions on $T$, and possibly on $X$, ensuring that $T$ admits an \emph{invariant measure} $m$ with some interesting property. 
(Note that since $T(0)=0$, the Dirac mass $\delta_0$ is an invariant measure for $T$, arguably not very interesting). For example, one may look for an invariant measure $m$  with \emph{full support} (\mbox{i.e.} such that $m(U)>0$ for every open set $U\neq\emptyset$), or even an invariant measure $m$ with full support with respect to which $T$ has some \emph{ergodicity} property, 
\mbox{e.g.} ergodicity in the usual sense, weak mixing or strong mixing. This kind of questions goes back at least to the classical work of Oxtoby and Ulam \cite{OU}. 

\smallskip
All \mea s considered in this paper will be finite Borel \mea s, and we will most of the time consider only \prob\ \mea s without writing explicitly so. We will denote by $\mathcal P_T(X)$ the set of all $T$-$\,$invariant Borel probability measures on $X$.  

\smallskip
The study of ergodic properties of linear dynamical systems was launched by Flytzanis in the paper \cite{Fl} and pursued in 
\cite{BG1}, leading to the obtention of necessary and sufficient conditions for operators on complex Hilbert spaces to admit \emph{Gaussian} invariant measures with one of the above properties. (A Borel probability measure on a complex Banach space $X$ is said to be Gaussian if every continuous linear functional $x^*\in X^*$ has a complex Gaussian distribution when considered as a random variable 
on $(X,{\rm Bor}(X),m)$; see \mbox{e.g} \cite[Chapter 5]{BM1} for more on Gaussian measures). The quest for similar results for operators acting on general Banach spaces was begun in \cite{BG2} and culminated in \cite{BM2}, where a very general condition (valid on any complex Banach space $X$) for an operator to admit a Gaussian ergodic measure with full support was obtained. 

Since we shall refer to this result from \cite{BM2} below and use some relevant terminology, we now state it precisely. Assume that $X$ is a complex Banach space. By a \emph{unimodular eigenvector} for an operator $T\in\mathfrak L(X)$, we mean an eigenvector whose associated eigenvalue has modulus $1$. We say that $T$ has a \emph{perfectly spanning set of unimodular eigenvectors} if, for every countable set $D\subset\TT$, we have 
$$\overline{\rm span}\,\Bigl[ \ker(T-\lambda I);\; \lambda\in\TT\setminus D\Bigr]= X.$$
This notion was introduced in \cite{Fl} in a Hilbert space setting. A formally stronger property was considered in \cite{BG1}, and it was shown later on in \cite{G} that the two notions are in fact equivalent. The aforementioned result from 
\cite{BM2} states that any operator with a perfectly spanning set of unimodular eigenvectors admits a Gaussian ergodic measure with full support. Without the requirement that the measure should be Gaussian, this was essentially proved earlier in \cite{G}.

\par\smallskip
We refer the reader to the book \cite{BM1} for more about linear dynamical systems, both from the topological and the ergodic theoretical points of view. Other interesting references 
are the book \cite{GEP} (which is more concerned with hypercyclicity issues), and the recent paper \cite{DG} for its point of view on linear systems as special cases of Polish dynamical systems.
\par\smallskip
In the present paper, we shall consider measure-theoretic properties of linear operators outside of the Gaussian framework. 

\smallskip
Our main interest lies with the so-called \emph{frequently hypercyclic} operators. Frequent hypercyclicity, which was introduced in \cite{BG1}, is a strengthening of  hypercyclicity which quantifies the frequency with which the orbit of a vector $x\in X$ visits a given nonempty open set.

The precise definition reads as follows. Let $T\in\mathfrak L(X)$. For any set $B\subset X$ and any $x\in X$, let us set $$\mathcal N_T(x,B):=\{ n\in\mathbb N;\; T^nx\in B\}\, .$$

The operator $T$ is said to be \emph{frequently hypercyclic} if there exists a vector $x\in X$ such that for each nonempty 
open set $V\subset X$, the set of integers 
$\mathcal N_T(x,V)$
has positive lower density; in other words,
$$\liminf_{N\to\infty}\frac1N \#\{ n\in [1,N];\; T^nx\in V\} >0\, .$$

Such a vector $x$ is said to be a frequently hypercyclic vector for $T$, and the set of all frequently hypercyclic vectors is denoted by $FHC(T)$. 

\smallskip
We have defined frequent hypercyclicity for linear dynamical systems only, but the notion obviously makes sense for an arbitrary Polish dynamical system $(X,T)$. In fact, in the 
\emph{compact} case such systems were studied long before frequently hypercyclic operators -- see \mbox{e.g.} Furstenberg's book \cite{Fu}. Likewise, what could be called hypercyclic systems (i.e. Polish dynamical systems admitting dense orbits) are of course central objects in topological dynamics. In this paper, we will use the linear terminology (hypercyclic, frequently hypercyclic) for both linear dynamical systems and general Polish dynamical systems.

\subsection{Two basic questions}\label{basic}

\par\smallskip
One of the interests of frequent hypercyclicity is that, although its definition is purely topological, its is naturally and deeply linked to measure-theoretic considerations about the dynamical system $(X,T)$. This is for instance testified by the following two facts.

\smallskip
- If $(X,T)$ is a Polish dynamical system and if $T$ admits an ergodic measure $m$ with full support, then $T$ is frequently hypercyclic and $FHC(T)$ has full $m\,$-$\,$measure. This 
follows easily from the pointwise ergodic theorem.

- If $(X,T)$ is a \emph{compact} dynamical system and if $T$ is frequently hypercyclic, then $T$ admits an invariant \prob\ measure with full support. A proof of this statement can be found \mbox{e.g.} in \cite{Fu}; it will also be briefly recalled at the beginning of Section 2.
\par\smallskip
These observations make it natural to wonder whether an arbitrary frequently hypercyclic Polish dynamical system always admits an invariant \prob\ measure with full support. That the answer to this question is negative was kindly pointed out to us by B. Weiss. Let $\rho$ be any irrational rotation of $\T$, and let $m$ be the normalized Lebesgue \mea\ on $\T$. 
Then $(\T, \rho)$ is a uniquely \erg\ dynamical system whose unique invariant probability measure is $m$ and for which all points are \fhy. Let $C$ be a nowhere dense compact subset of $\T$ such that $m(C)>0$, and set $X:=\T\setminus\bigcup_{n\in\Z}\rho^{-n}(C)$. Then $X$ is a $\rho\,$-$\,$\inv, dense $G_{\delta }$ subset of $\T$, and hence $(X,\rho)$ is a \fhy\ Polish dynamical system. On the other hand, we have $m(X)=0$ because $m(X)<1$ and $m$ is an ergodic measure for $(\TT, \rho)$.  Since $\rho$ is uniquely \erg, it follows that $(X,\rho)$ admits no \inv\ \mea\ at all.
\par\smallskip
This example is however highly non linear, and one may naturally ask what happens in the linear setting.

\begin{question}\label{Q1} Let $(X,T)$ be a frequently hypercyclic linear dynamical system. Is it true that $T$ admits an invariant \prob\ measure with full support? 
\end{question}

One may also ask whether a stronger result holds true:

\begin{question}\label{Q2} Let $(X,T)$ be a frequently hypercyclic linear dynamical system. Is it true that $T$ admits an \emph{ergodic} \prob\ measure with full support? 
\end{question}

\smallskip
Our first main result is concerned with Question \ref{Q1}. We obtain a positive answer when $T$ is a linear operator acting on a \emph{reflexive} Banach space $X$. 

\smallskip
Recall that a measure $m$ on $X$ is said to be \emph{continuous} if $m(\{ a\})=0$ for every $a\in X$. Denoting by ${\rm Per}(T)$ the set of all \emph{periodic points} of $T$, it is not hard to see that if $m$ is a $T\,$-$\,$invariant measure such that $m({\rm Per}(T))=0$, then $m$ is necessarily continuous.

\begin{theorem}\label{th1} If $X$ is a reflexive Banach space, then any frequently hypercyclic operator $T$ on $X$ admits a continuous invariant \prob\ measure $m$ with full support. In fact, one may require that the measure $m$ satisfies 
$m({\rm Per} (T))=0$. 
\end{theorem}

Applying the \emph{ergodic decomposition theorem} (see \mbox{e.g} \cite[Ch. $2$, Sec. $2.2$]{Aa}), we easily deduce the following result:
\begin{corollary} If $X$ is a reflexive Banach space and if $V\subset X$ is a nonempty open set, then any frequently hypercyclic operator $T$ on $X$ admits a continuous ergodic probability measure $\mu$ such that $\mu(V)>0$. 
\end{corollary}

\smallskip
Indeed, let $m$ be an invariant measure with full support for $T$ such that $m({\rm Per}(T))=0$. By the ergodic decomposition theorem, one may write 
$$m=\int_S \mu^s \,d\mathbf p(s)\, ,$$
where the $\mu^s$ are ergodic probability measures for $T$ and  the integral is taken over some probability space $(S,\mathbf p)$. The meaning of the above formula is that 
$m(A)=\int_S \mu^s(A)d\mathbf p(s)$ for every Borel set $A\subset X$. As $m$ has full support, we have 
$m(V)>0$ and hence $\mathbf p\left(\{ s\in S; \mu^s(V)>0\}\right)>0$; and since $m({\rm Per}(T))=0$ we have $\mu^s({\rm Per}(T))=0$ for $\mathbf p\,$-$\,$almost every $s\in S$. So one can find $s$ such that the measure $\mu=\mu^s$ has the required properties.

\medskip
The existence of an invariant measure with full support in Theorem \ref{th1} relies on the following ``nonlinear" statement, which is valid for a rather large class  of Polish dynamical systems.  

\begin{theorem}\label{th2} Let $(X,T)$ be a Polish dynamical system. Assume that $X$ is endowed with a topology $\tau$ which is Hausdorff, coarser than 
the original topology $\tau_X$, and such that every point of $X$ has a neighbourhood basis with respect to $\tau_X$ consisting of $\tau\,$-$\,$compact sets. Moreover, assume that
\begin{itemize}
\item[\rm (i)] $T$ is frequently hypercyclic with respect to $\tau_X$;
\item[\rm (ii)] $T$ is a continuous self-map of $(X,\tau)$.
\end{itemize}
Then $T$ admits an invariant \prob\ measure with full support.
\end{theorem}

Forgetting the continuity requirement on $m$, the first part of Theorem \ref{th1} follows at once from this result, taking for $\tau$ the weak topology of the reflexive Banach space 
$X$ and using the fact that the closed balls of $X$ are weakly compact.

\smallskip
Since compactness is rather crucial in all our arguments, and although we are not able to provide a counterexample, it seems reasonable to believe that Theorem \ref{th1} breaks down when the space $X$ is not assumed to be reflexive. Good candidates for counterexamples could be bilateral weighted shifts on $c_0(\mathbb Z)$ in the spirit of those constructed by Bayart and Rusza in \cite{BR}. So we restate Question \ref{Q1} as

\begin{question}\label{Q4} Does there exist a frequently hypercyclic operator (necessarily living on some non-reflexive Banach space) which does not admit any invariant measure with full support?
\end{question}

\smallskip
Regarding Question \ref{Q2}, we are able to answer it in the negative. Let us first note that this question was already investigated in \cite{BG2}, where it was proved that there exists a bounded operator on the sequence space $X=c_0(\mathbb Z_+)$ (actually, a weighted backward shift) which is frequently hypercyclic but does not admit any ergodic \emph{Gaussian} measure with full support. The proof given in \cite{BG2} can be easily extended to show that this frequently hypercyclic operator does not admit any ergodic measure $m$ with full support such that $\int_X \Vert x\Vert^2 dm(x)<\infty$; but the existence of a second order moment seems to be crucial in the argument.
 Our second main results shows that this aditional assumption can be dispensed with.

\begin{theorem}\label{th3} There exists a frequently hypercyclic operator on the space $c_0(\mathbb Z)$ which does not admit any ergodic measure with full support.
\end{theorem}

However, Question \ref{Q2} remains widely open on all other ``classical" Banach spaces, and in particular for Hilbert space operators.  
In this special case there are several reasons to believe that the answer could be ``Yes", some of which will be outlined in the sequel. More generally, Theorem \ref{th2} 
makes the following instance of Question \ref{Q2} especially interesting.  

\begin{question}\label{Q6} Does any frequently hypercyclic operator $T$ acting on a reflexive Banach space admit an ergodic \prob\ measure with full support? 
\end{question}

\subsection{A natural parameter} 
One of the main ingredients in the proof of Theorem \ref{th3} is a certain parameter $c(T)\in [0,1]$ associated with any hypercyclic operator $T$ acting on a Banach space $X$. 
Very roughly speaking, $c(T)$ is the maximal frequency with which the orbit of a hypercyclic vector $x$ for $T$ can visit a ball centered at $0$. In fact, for any $\alpha>0$ we have 
$$c(T)=\sup\,\left\{ c\geq 0;\; \overline{\rm dens}\,(\mathcal N_T(x,B(0,\alpha))\geq c\;\hbox{ for comeager many}\; x\in HC(T)\right\}. $$
  
The parameter $c(T)$ can bear witness of the existence of an ergodic measure with full support. Indeed, we can prove

\begin{theorem}\label{th4} Let $(X,T)$ be a linear dynamical system. If $T$ admits an ergodic measure with full support, then $c(T)=1$.
\end{theorem}

Theorem \ref{th3} follows from Theorem \ref{th4} combined with a result from \cite{BR} which (although not stated in this form) yields the existence of frequently hypercyclic bilateral weighted shifts $B_{\bf w}$ on $c_0(\mathbb Z)$ such that $c(B_{\bf w})<1$.

\par\smallskip
As it turns out, the parameter $c(T)$ is also closely connected to the existence of \emph{distributionally irregular vectors} for the operator $T$. A vector $x\in X$ is said to be distributionally irregular for $T$ if there exist two sets of integers $A,B\subset\mathbb N$, both having upper density $1$, such that 
$$\displaylines{\Vert T^ix\Vert\xrightarrow[i\to\infty\atop i\in A]{}0 \\ {\rm and}\\ \Vert T^ix\Vert\xrightarrow[i\to\infty\atop i\in B]{} \infty.}$$
This notion was studied by Bernardes, Bonilla, M\"uller and Peris in  \cite{BBMP}, where it is shown that the existence of a distributionally irregular vector for a linear system $(X,T)$ is equivalent to this system being \emph{distributionally chaotic} in the sense of Schweitzer and Sm\'ital \cite{SS}.  
\par\smallskip
In \cite{BBMP}, the authors ask (a restricted form of) the following question. Assume that $T$ 
is a bounded operator on a complex Banach space $X$ and that $T$ has a perfectly spanning set of unimodular eigenvectors. Is it then true that $T$ admits a distributionally irregular vector? Since any such operator admits an ergodic measure with full support, the following result (whose proof makes use of the parameter $c(T)$ in a crucial way) answers this question in the affirmative.

\begin{theorem}\label{th5} Let $(X,T)$ be a linear dynamical system, and assume that $T$ admits an ergodic measure with full support. Then $T$ admits a comeager set of distributionally irregular vectors. This holds in particular if $X$ is a complex Banach space and $T$ has a perfectly spanning set of unimodular eigenvectors.
\end{theorem}

\smallskip
The parameter $c(T)$ already appears implicitely in the proof of Theorem \ref{th1} above, and it will be clear from this proof that its exact value is not easy to determine. In view of Theorem \ref{th4}, the following question is quite natural. Note that a positive answer to this question, together with the result from \cite{BR} mentioned above, would answer Question \ref{Q4} affirmatively.

\begin{question}\label{enervante} Let $(X,T)$ be a linear dynamical system. Is it true that if $T$ is hypercyclic and admits an invariant measure with full support, then $c(T)=1$?
\end{question}

 \smallskip
The following special case is worth stating separately.

\begin{question}\label{Q5} Let $T$ be a frequently hypercyclic operator acting on a reflexive Banach space $X$. Is $c(T)$ necessarily equal to $1$?
\end{question}

It is tempting to conjecture that the answer to this question is positive. In any event, since any \op\ $T$ admitting an \erg\ \mea\ with full support satisfies $c(T)=1$ by Theorem \ref{th4}, this is in some sense a ``necessary first step" towards a possibly affirmative answer to Question \ref{Q6}.

\subsection{Baire category arguments} 
The following simple observation will be used throughout the paper: to obtain an ergodic measure with full support for an operator $T$, it is in fact enough to find an invariant measure $m$ for $T$ such that $m (HC(T))>0$. Indeed, the ergodic decomposition theorem then yields the existence of an \emph{ergodic} measure $\mu$ such that  $\mu(HC(T))>0$.  Being $T\,$-$\,$invariant, this measure necessarily has full support because 
$HC(T)\subset \bigcup_{n\geq 0} T^{-n}(U)$ for each nonempty open set $U$. 
\par\smallskip
One of the most tempting roads one could follow in order to find an invariant measure $m$ such that $m(HC(T))>0$ is to try to use Baire Category arguments in the space of $T$-$\,$invariant measures $\mathcal P_T(X)$. This is far from being a new idea. Indeed, this strategy has already proved to be quite successful in a number of nonlinear situations; see for instance \cite{Si1}, \cite{Si2} or \cite{CS}. 

\smallskip
Recall that for any Polish space $X$, the space $\mathcal P(X)$ of all Borel \prob\ \mea s on $X$ is endowed with the so-called \emph{Prokhorov topology}, which is the weak topology generated by the space $\mathcal C_b(X)$ of all real-valued, bounded continuous functions on $X$. In other words, a sequence $(\mu_n)$ of elements of 
$\mathcal P(X)$ converges to $\mu$ in $\mathcal P(X)$ if and only if $\int_X f\, d\mu_n\to\int_X f\, d\mu$ for every $f\in\mathcal C_b(X)$. The topology of $\mathcal P(X)$ is Polish because $X$ is Polish (see \mbox{e.g.} \cite[Chapter 2]{Bo}). The set 
$\mathcal P_T(X)$ of all $T$-$\,$\inv\ Borel \prob\ \mea s on $X$ is easily seen to be closed in $\mathcal P(X)$, and hence  it is a Polish space in its own right.

\smallskip
For any Borel set $A\subset X$, we denote by $\mathcal P(A)$ the set of of all 
Borel probability measures $m$ on $X$ which are supported on $A$ (\mbox{i.e.} such that $m(A)=1$) and we set $\mathcal P_T(A):=\mathcal P_T(X)\cap\mathcal P(A)$.  If $O\subset X$ is open, then $\mathcal P(O)$ is easily seen to be $G_\delta$ in $\mathcal P(X)$; so $\mathcal P_T(O)$ is $G_\delta$ in $\mathcal P_T(X)$. 

\smallskip
Let us say that a Borel set $A\subset X$ is \emph{backward $T$-$\,$invariant} if $T^{-1}(A)\subset A$. One can write $HC(T)$ as a countable intersection of backward $T$-$\,$invariant open sets $O_j$ (just set $O_j:=\bigcup_{n\geq 0} T^{-n}(V_j)$, where $(V_j)_{j\geq 1}$ is a countable basis of nonempty open sets for $X$) so that $\mathcal P_T(HC(T))=\bigcap_{j\geq 1} \mathcal P_T(O_j)$. Hence, a positive answer to the next question would solve Question \ref{Q6} affirmatively.

\begin{question}\label{Q7} Let $T$ be a frequently hypercyclic operator on a reflexive Banach space $X$. Is it true that for every backward $T$-$\,$invariant, nonempty open 
$O\subset X$, the set $\mathcal P_T(O)$ is dense in $\mathcal P_T(X)$? Equivalently, is $\mathcal P_T(HC(T))$ dense in $\mathcal P_T(X)$?
\end{question}

The two formulations of the question are indeed equivalent, because every backward $T$-$\,$invariant open $O\neq\emptyset$ contains $HC(T)$.

\smallskip
This very same question may be considered for \emph{chaotic} operators. Recall that an operator $T$ on $X$ is said to be chaotic if it is hypercyclic and its periodic points are dense in $X$. One of the most exciting open problems in linear dynamics is to determine whether every chaotic operator is frequently hypercyclic. This is widely open even in the Hilbert space setting, and closely related to an older question of Flytzanis \cite{Fl} asking whether a hypercyclic operator on a Hilbert space $H$ whose unimodular eigenvectors of $T$ span a dense linear subspace of $H$ necessarily has uncountably many unimodular eigenvalues (or even a perfectly spanning set of unimodular eigenvectors). A positive answer to the next question would imply that in fact, every chaotic operator admits an ergodic measure with full support.

\begin{question}\label{Q8} Let $T$ be a chaotic operator on a Banach space $X$. Is it true that $\mathcal P_T(O)$ is dense in 
$\mathcal P_T(X)$ for any backward $T$-$\,$invariant open set $O\neq X$? Equivalently, is
$\mathcal P_T(HC(T))$ dense in $\mathcal P_T(X)$?
\end{question}

It is worthing pointing out that the corresponding statement is known to 
\emph{fail} in the nonlinear setting. Indeed, an example is given in \cite{W} of a compact dynamical system $(X,T)$ with $T$ invertible, such that $T$ is topologically transitive with a dense 
set of periodic points, but admits no ergodic measure with full support. What makes this example especially interesting is that the map $T$ is in fact not frequently 
hypercyclic. This leads naturally to the following intriguing question.

\begin{question}\label{Q10} Let $(X,T)$ be a compact dynamical system, and assume that $T$ is frequently hypercyclic. Does it follow that $T$ admits an ergodic measure with full support?
\end{question}

\par\smallskip
Our last result, Theorem \ref{th6} below, shows in particular that a weak form of Question \ref{Q8} does have a positive answer for a large class of chaotic operators which, to the best of our knowledge, contains all known concrete chaotic operators. Note however that Theorem \ref{th6} cannot be of any use for showing that every chaotic operator has an ergodic measure with full support, since the assumption  made therein that the \op\ $T$ has a perfectly spanning set of unimodular eigenvectors already implies the existence of such a measure.
\par\smallskip
Let us say that a measure $\nu\in\mathcal P(X)$ is a \emph{periodic measure for} $T$ if it has the form $$\nu=\frac{1}N\sum\limits_{n=0}^{N-1} \delta_{T^n a},$$ where $a\in X$ and $N\ge 1$ satisfy $T^Na=a$. We will denote by ${\mathcal F}_T(X)$ the convex hull of the set of all 
 periodic measures for $T$. Equivalently, ${\mathcal F}_T(X)$ is the the set of all $T$-$\,$invariant, finitely supported measures (which explains the notation).   The closure of 
 ${\mathcal F}_T(X)$ in $\mathcal P_T(X)$ is denoted by $\overline{\mathcal F}_T(X)$, and for any Borel set $A\subset X$ we set ${\overline{\mathcal F}}_T(A):={\overline{\mathcal F}}_T(X)\cap\mathcal P(A).$
 
Another family of invariant measures will be of interest for us. We shall say that a measure $\mu\in\mathcal P(X)$ is a 
\emph{Steinhaus measure for $T$} if $\mu$ is the distribution of a random variable $\Phi:(\Omega ,\mathfrak{F}, \P)\to X$ defined on some standard probability space 
$(\Omega ,\mathfrak{F}, \P)$, of the form
$$\Phi(\omega)=\sum_{j\in J} \chi_j(\omega)\, x_j\, ,$$
where the $x_j$ are unimodular eigenvectors for $T$ and $(\chi_j)_{j\in J}$ is a finite sequence of independent {Steinhaus variables} (\mbox{i.e.} random variables uniformly distributed on the circle $\TT$). Any Steinhaus measure for $T$ is $T\,$-$\,$invariant, by the rotational invariance of the Steinhaus variables. We denote by $\mathcal S_T(X)$ the family of all Steinhaus measures for $T$, and by $\overline{\mathcal S}_T(X)$ the closure of ${\mathcal S}_T(X)$ in $\mathcal P(X)$. Accordingly, we set 
$\overline{\mathcal S}_T(A)=:\overline{\mathcal S}_T(X)\cap\mathcal P(A)$ for any Borel set $A\subset X$.

\begin{theorem}\label{th6} Let $T$ be a bounded operator on a complex separable Banach space $X$.  

\begin{itemize}
\item[\rm (a)] Assume that $T$ has a perfectly spanning set of unimodular eigenvectors, and that 
the periodic eigenvectors of $T$ are dense in the set of all unimodular eigenvectors. Then  
${\overline{\mathcal F}}_T(HC(T))$ is a dense $G_{\delta }$ subset of ${\overline{\mathcal F}}_T(X)$.
\item[\rm (b)] Assume ``only" that $T$ has a perfectly spanning set of unimodular eigenvectors. Then  
${\overline{\mathcal S}}_T(HC(T))$ is a dense $G_{\delta }$ subset of ${\overline{\mathcal S}}_T(X)$.
\end{itemize}
\end{theorem}

As explained above, the existence of an invariant measure supported on $HC(T)$ implies (and in fact, is equivalent to) the existence of an ergodic measure with full support. Hence, it follows in particular from part (b) that any operator with a perfectly spanning set of unimodular eigenvectors admits an ergodic measure with full support. This result is of course weaker than the one obtained in \cite{BM2} since the ergodic measure has no reason for being Gaussian; but the the \emph{proof} is quite different (being based on the Baire category theorem) and it looks much simpler than the existing ones from \cite{G} and \cite{BM2}.

\subsection{Organization of the paper}
Section 1 is purely ``nonlinear". We first prove Theorem \ref{th2}, and then we add a few simple remarks. In particular, it is shown in Proposition \ref{propcontinuous} that under rather mild assumptions, the existence of an invariant measure with full support implies the existence of a \emph{continuous} measure with these properties. 
Theorem \ref{th1} is proved in Section \ref{sec3}. In the same section (Proposition \ref{continuous2}), we also show that if an operator $T\in\mathfrak L(X)$ admits an invariant measure with full support, then the continuous, $T\,$-$\,$invariant measures with full support form a dense $G_\delta$ subset of $\mathcal P_T(X)$. The parameter $c(T)$ is introduced in Section 
\ref{sec4}. This allows us to prove Theorems \ref{th3}, \ref{th4} and \ref{th5} quite easily,  together with two simple additional results: any frequently hypercyclic operator has a comeager set of ``distributionally null" orbits; and the set of all frequently hypercyclic vectors for a given operator $T$ is always meager in the underlying Banach space $X$ (this was obtained independently in \cite{BR} and \cite{M}). The proof of Theorem \ref{th6} is given in Section \ref{sec5}. It makes use of another result, Theorem \ref{th7}, which provides several necessary and sufficient conditions for the existence of an invariant measure supported on $HC(T)$ and belonging to the closure of a family of invariant measures satisfying some natural assumptions. Using Theorem \ref{th7}, we also prove two additional results similar to Theorem \ref{th6} which give some plausibility to the conjecture that every chaotic operator is frequently hypercyclic and in fact admits an ergodic measure with full support. We conclude the paper by listing several equivalent formulations of the perfect spanning property.

\section{Construction of invariant measures with full support}\label{sec2}

In this section, we prove Theorem \ref{th2}. As already mentioned in the introduction, it is a well-known fact that frequently hypercyclic continuous self-maps of a compact metric space admit invariant measures with full support. Since the proof of Theorem \ref{th2} uses in a crucial way the idea of the proof in the compact case, we first sketch the latter briefly. We refer to 
\cite[Lemma 3.17]{Fu} for more details.

\subsection{The compact case}\label{compactcase} 
Let $T$ be a frequently hypercyclic continuous self-map of a compact metrizable space $X$, and let $x_0\in FHC(T)$. Denote by $\mathcal C(X)$ the 
space of real-valued continuous functions on $X$. 
By the Riesz representation theorem, one can identify 
$\mathcal P(X)$ with the set of all positive linear functionals $L$ on $\mathcal C(X)$ such that $L(\mathbf 1)=1$, where $\mathbf 1$ denotes the function constantly equal to $1$. The latter is $w^*$-$\,$compact as a subset of $\mathcal{C}(X)^{*}$, and it is also metrizable because 
$\mathcal C(X)$ is separable. For each $N\in\mathbb N$, let $\mu _{N}\in \mathcal{P}(X)$ be defined as
$$\mu_N:=\frac1N\sum_{n=1}^N \delta_{T^nx_0}\, ,$$
where, for each $a\in X$, $\delta_a$ is the Dirac mass at $a$. Since all $\mu_N$ are probability measures, one can find an increasing sequence of integers $(N_k)_{k\ge 1}$ and a probability measure $m\in\mathcal P(X)$ such that $\mu_{N_k}$ tends to $ m$ in the $w^{*}$-$\,$topology of $\mathcal{P}(X)$ as $k$ tends to infinity; in other words, 
$$\frac1{N_k} \sum\limits_{n=1}^{N_k} f(T^nx_0)\To \int_X f\, dm\quad\textrm{ for every }\;
f\in\mathcal C(X).$$ 
\par\smallskip
Since $\int_X (f\circ T)\, d\mu_N=\frac{1}{N}\sum\limits_{n=2}^{N+1} f(T^nx_0)$ for any $N\in\NN$, we see that 
$$\int_X (f\circ T)\, d\mu_{N_{k}}-\int_X f\, d\mu_{N_{k}}\To 0 \quad\textrm{ for every }\;
f\in\mathcal C(X).$$ It follows 
that $\int_X (f\circ T)\, dm=\int_X f\, dm$ for every $f\in\mathcal C(X)$, so the measure $m$ is $T$-$\,$invariant.
\par\smallskip
Let $U$ be a nonempty open set in $X$, and choose a nonempty open set $V$ whose closure $\overline V$ is contained in $ U$. Since $\overline V$ is closed in $X$, the map 
$\mu\mapsto \mu(\overline V)$ is upper semi-continuous on $(\mathcal P(X),w^*)$; so we have
$$m(\overline V)\geq \limsup_{k\to\infty} \mu_{N_k}(\overline V).$$
Since $\mu_{N_k}(\overline V)=\frac1{N_k} \#\left\{ n\in[1,N_k];\; T^n x_0\in\overline V\right\}$, and recalling that $x_0\in FHC(T)$, it follows that 
$$m(U)\geq \liminf_{N\to\infty} \frac1N\#\left\{ n\in [1,N]  \textrm{ ; } T^nx_0\in V\right\}>0\, .$$
This shows that the measure $m$ has full support.

\medskip
It is clear that compactness is crucial in the above proof, since essentially everything relies on the Riesz representation theorem. The metrizability of $X$ is also needed for two reasons: it implies that $\mathcal C(X)$ is separable, so that we can extract from the sequence $(\mu _{N})_{N\ge 1}$ a $w^*$-$\,$convergent 
\emph{sequence} $(\mu_{N_k})_{k\ge 1}$ (but this is merely a matter of convenience); and it allows to identify the linear functionals on $\mathcal C(X)$ with the \emph{Borel} measures on $X$.

\subsection{Proof of Theorem \ref{th2}}
In the proof of Theorem \ref{th2}, we will use the same idea as above to associate with each $\tau\,$-$\,$compact set $K$ a Borel probability measure 
$\mu_K$ on $K$. Then the measure $m$ will be obtained as the supremum of all these measures $\mu_K$. In what follows, we denote by $\mathcal K_\tau$ the family of all 
$\tau\,$-$\,$compact subsets of $X$. We also fix a frequently hypercyclic point $x_0$ for $T$.

\par\smallskip
Since we will have to consider simultaneously all sets $K\in\mathcal K_\tau$, a diagonalization procedure will be needed. To avoid extracting infinitely many sequences of integers, it is convenient to consider a suitable \emph{invariant mean} on $\ell^\infty(\mathbb N)$, the space of all bounded sequences of real numbers. Recall that an invariant mean is a positive linear functional $\mathfrak m$ on $\ell^\infty(\mathbb N)$ such that $\mathfrak m(\mathbf 1)=1$ and 
$\mathfrak m\bigl(\phi(\,\cdot +a)\bigr)=\mathfrak m(\phi)$ for every $a\in\mathbb N$ and all $\phi=(\phi(i))_{i\geq 1}\in\ell^\infty(\mathbb N)$, where $\phi(\,\cdot +a)$ is the translated sequence $\bigl(\phi(i+a)\bigr)_{i\geq 1}$. 
\par\smallskip
It is not hard to see that there exists an invariant mean $\mathfrak m$ such that 
$$\mathfrak m(\phi)\geq \liminf_{n\to\infty} \frac1n\sum_{i=1}^n \phi(i)\quad \textrm{ for all }\phi\in\ell^\infty(\mathbb N)\,.$$
 For example, one may take $\mathfrak m(\phi)=\lim_{\mathcal U} \frac1n\sum\limits_{i=1}^n \phi(i)$, where $\mathcal U$ is a non-principal ultrafilter on 
$\mathbb N$. In the sequel, we fix such an invariant mean $\mathfrak m$. In order to emphasize the fact that $\mathfrak m$ should be viewed as a finitely additive measure on $\NN$, we write the result of the action of $\mathfrak m$ on a ``function" $\phi\in\ell^\infty(\NN)$ as an integral:
$$\mathfrak m(\phi)=\int_{\mathbb N} \phi(i)\, d\mathfrak m(i)\, .$$
So the invariance property reads
$$\int_\NN \phi(i+a)\, d\mathfrak m(i)=\int_\NN \phi(i)\, d\mathfrak m (i)\quad \textrm{ for every }a\in\N\,.$$
\par\smallskip
Before really starting the proof of Theorem \ref{th2}, let us observe that the topologies $\tau_X$ and $\tau$ have the same Borel sets. Indeed, since each point $x\in X$ has a neighbourhood basis consisting of $\tau\,$-$\,$compact sets and since the topology $\tau_X$ is Lindel\"of (being separable and metrizable), every $\tau_X\,$-$\,$open set is a countable unions of $\tau\,$-$\,$compact sets and hence is $\tau\,$-$\,$Borel. So it makes sense to speak of Borel measures on $X$ without referring explicitly to one of the topologies $\tau_X$ or $\tau$.

\par\smallskip
We now start the proof of Theorem \ref{th2} with the following fact, that will allow us to deal with \emph{Borel} measures when applying the Riesz representation theorem.

\begin{fact}\label{fact1} Every $\tau\,$-$\,$compact subset of $X$ is $\tau\,$-$\,$metrizable.
\end{fact}
\begin{proof} Let us fix $K\in\mathcal K_\tau$, and let $(V_j)_{j\geq 1}$ be a countable basis of (nonempty) open sets for $K$ with respect to the topology $\tau_X$.
 For each $j\geq 1$, the $\tau\,$-$\,$closure $E_j$ of $V_j$ is $\tau\,$-$\,$compact. For each pair $\mathbf j=(j_1,j_2)$ with $E_{j_1}\cap E_{j_2}=\emptyset$, one can find a 
 $\tau\,$-$\,$continuous function $f_{\mathbf j}:K\to \RR$ such that $f_{\bf j}\equiv 1$ on $E_{j_1}$ and $f_{\bf j}\equiv 0$ on $E_{j_{2}}$. It is clear that the (countable) family of all 
 such functions $f_{\bf j}$ separates the points of $K$; and since $K$ is $\tau\,$-$\,$compact, metrizability follows.\end{proof}
 
 \par\smallskip
 For each $K\in\mathcal K_\tau$, let us denote by $\mathcal C(K,\tau)$ the space of all $\tau\,$-$\,$continuous, real-valued functions on $K$. Using the same argument as in 
 Section \ref{compactcase} above, we can now prove
 
 \begin{fact}\label{fact2} For every $K\in\mathcal K_\tau$, there exists a unique positive Borel measure $\mu_K$ on $K$ such that 
 $$\int_K f\, d\mu_K=\int_\NN (\mathbf 1_Kf)(T^i x_0)\, d\mathfrak m(i)\quad
 \textrm{ for every } f\in \mathcal C(K,\tau).$$ The measure $\mu_K$ satisfies $0\leq \mu_K(K)\leq 1$. Moreover, if $K$ has nonempty interior with respect to the topology $\tau_X$, 
 then $\mu_K(K)>0$.
 \end{fact}
 
 \begin{proof} The first part is obvious by the Riesz representation theorem, since the formula 
 $$L(f)=\int_\NN (\mathbf 1_Kf)(T^i x_0)\, d\mathfrak m(i), \;\; f\in \mathcal C(K,\tau)$$ defines a positive linear functional on $\mathcal C(K,\tau)$. It is also clear that  the 
 \mea\ $\mu _{K}$ thus defined satisfies $\mu_K(K)\leq \mathfrak m(\mathbf 1)=1$. Now, let us denote by $V$ the $\tau_X\,$-$\,$interior of $K$ in $X$, and assume 
 that $V\neq\emptyset$. Then 
 \begin{eqnarray*}\mu_K(K)&=&\int_\NN \mathbf 1_K(T^ix_0)\,d\mathfrak m(i )
 \geq \liminf_{n\to\infty}\frac 1n\sum_{i=1}^n \mathbf 1_K(T^ix_0)
 \geq \liminf_{n\to\infty}\frac 1n\sum_{i=1}^n \mathbf 1_V(T^ix_0)
 >0
 \end{eqnarray*} because $x_0$ is frequently hypercyclic for $T$.
 \end{proof}
 
\par\smallskip
For each $K\in\mathcal K_\tau$, we extend the measure $\mu_K$ to a positive Borel measure on $X$ (still denoted by $\mu_K$) in the usual way; that is,  we set 
$\mu_K(A):=\mu_K(K\cap A)$ 
for every Borel set $A\subset X$. Then $\mu_K(X)\leq 1$.
\par\smallskip
The following simple yet crucial fact will allow us to define a measure $m$ on $X$ as the supremum of all measures $\mu_K$.

\begin{fact}\label{fact3} If $K,L\in\mathcal K_\tau$ and if $K\subset L$, then $\mu_K\leq\mu_L$.
\end{fact}

\begin{proof} Since $(X,\tau_X)$ is Polish, every Borel measure on $X$ is regular. So it is enough to show that $\mu_K(E)\leq\mu_L(E)$ for every $\tau_X\,$-$\,$compact set 
$E\subset X$. We do this in fact for every $\tau\,$-$\,$compact set $E$. (Recall that the topology $\tau_X$ is finer than $\tau$, so every $\tau_X\,$-$\,$compact set is 
$\tau\,$-$\,$compact).

Since $K\cap E$ is $\tau\,$-$\,$compact, we may replace $E$ with $K\cap E$, \mbox{i.e} we may assume that $E\subset K$. Since $E$ is $\tau\,$-$\,$compact, the function 
$\mathbf 1_E$ (the indicator function of $E$) is upper-semicontinuous with respect to the topology $\tau$, when considered as a function on $L$. So, by the metrizability of 
$(L,\tau)$, one can find a decreasing sequence $(f_j)_{j\geq 1}$ of functions of $\mathcal C(L,\tau)$ such that $f_j$ converges to $ \mathbf 1_E$ pointwise on $L$. Of course, the restrictions of the functions $f_j$ to $K$ belong to $\mathcal C(K,\tau)$. Since $K\subset L$ and $f_j\geq 0$ we have
\begin{eqnarray*}
\int_K f_j\, d\mu_K&=&\int_\NN (\mathbf 1_K f_j)(T^ix_0)\, d\mathfrak m(i)
\leq\int_\NN (\mathbf 1_L f_j)(T^ix_0)\, d\mathfrak m(i)
=\int_K f_j\, d\mu_L
\end{eqnarray*}
for all $j\geq 1$. Letting $j$ tend to infinity on both sides, we get $\mu_K(E)\leq \mu_L(E)$.
\end{proof}

\par\smallskip
From Fact \ref{fact3} and since the family of $\tau$-$\,$compact sets is closed under finite unions, we see that the family $(\mu_K)_{K\in\mathcal K_\tau}$ is what is sometimes called ``filtering increasing". From this, we can easily deduce 

\begin{fact}\label{fact4} If we set
$$\displaystyle \mu(A):=\sup_{K\in\mathcal K_\tau} \mu_K(A)\quad \textrm{ for every Borel set } A\subset X,$$ then $\mu $
is a positive Borel measure on $X$, such that $\mu(X)\leq 1$. 
\end{fact}

\begin{proof} Obviously $0\leq \mu(A)\leq 1$ for every Borel set $A\subset X$ and $\mu(X)>0$. It is also clear that $\mu\left(\bigcup_n A_n\right)=\sup_n \mu(A_n)=\lim \mu(A_n)$ for every increasing sequence of Borel sets $(A_n)$. So we just have to check that $\mu$ is finitely additive. 

Let $A,A'$ be Borel subsets of $X$ with $A\cap A'=\emptyset$. Since $\mu_K(A\cup A')=\mu_K(A)+\mu_K(A')\leq \mu (A)+\mu (A')$ for all $K\in \mathcal K_\tau$, we get that
$\mu(A\cup A')\leq \mu (A)+\mu (A')$ by the very definition of $\mu$. Conversely, we have by Fact 3 
$$\mu_K(A)+\mu_{K'}(A')\leq \mu_{K\cup K'} (A)+\mu_{K\cup K'} (A')=\mu_{K\cup K'}(A\cup A')\leq \mu(A\cup A')$$
 for any $K,K'\in\mathcal K_\tau$, and hence $\mu(A)+\mu(A')\leq \mu(A\cup A')$.
\end{proof}

We now check that $\mu$ has the required properties.

\begin{fact}\label{fact5} The measure $\mu$ is $T$-$\,$invariant and has full support.
\end{fact}

\begin{proof} The fact that $\mu$ has full support is obvious by Fact \ref{fact2}: if $U$ is a nonempty $\tau_X\,$-$\,$open subset of $X$, then $U$ contains a $\tau\,$-$\,$compact set $K$ 
with nonempty $\tau_X\,$-$\,$interior and hence $\mu(U)\geq \mu(K)\geq \mu_K(K)>0$.
\par\smallskip
The main point is to show that $\mu$ is $T$-$\,$invariant. For this, it is in fact enough to show that $\mu(T^{-1}(E))
\le\mu(E)$ for every $\tau\,$-$\,$compact set $E\subset X$. Suppose indeed that it is the case. Then, by the regularity of the Borel measures $\mu$ and $\mu\circ T^{-1}$, we have $\mu(T^{-1}(A))\le\mu(A)$ for every Borel set $A\subset X$. Applying this inequality to $X\setminus A$ now yields that $\mu(X\setminus T^{-1}(A))
\le\mu(X\setminus A)$, and since $\mu (X)<\infty$ this means that $\mu(T^{-1}(A))\ge\mu(A)$. So we get $\mu(T^{-1}(A))=\mu(A)$ for every Borel set $A\subset X$, which proves that $\mu $ is $T$-invariant.
\par\smallskip
Let us fix $E\in\mathcal K_\tau$. 
We are first going to prove that 
\begin{equation*}\label{formule} \mu_K(T^{-1}(E))\leq \mu_{T(K)} (E)\quad \textrm{ for every }K\in\mathcal K_\tau.
\end{equation*}
 Note that this makes sense because, as $T$ is continuous with respect to the topology $\tau$, $T(K)$ belongs to $\mathcal K_\tau$.
Since $E\cap T(K)$ is closed in $(T(K),\tau)$, there exists a decreasing sequence 
$(f_j)_{j\ge 1}$ of functions in $\mathcal C(T(K),\tau)$ such that $f_j$ converges to $\mathbf 1_{E\cap T(K)}$ pointwise on $T(K)$. Then we have
\begin{eqnarray*}
\mu_{T(K)}(E)&=&\int_{T(K)} \mathbf 1_{E\cap T(K)}\, d\mu_{T(K)}
=\lim_{j\to\infty} \int_{T(K)} f_j\, d\mu_{T(K)}\\
&=&\lim_{j\to\infty} \int_\NN (\mathbf 1_{T(K)} f_j)(T^ix_0)\, d\frak m(i)
=\lim_{j\to\infty} \int_\NN (\mathbf 1_{T(K)} f_j)(T^{i+1}x_0)\, d\frak m(i)\, ,
\end{eqnarray*}
the last equality being true because $\mathfrak m$ is an invariant mean on $\ell^{\infty}(\N)$.
\smallskip
Now observe that since $f_j$ is nonnegative on $T(K)$, we have
$$\bigl(\mathbf 1_{T(K)}f_j\bigr)(T^{i+1}x_0)\geq \bigl( \mathbf 1_K \cdot (f_j\circ T)\bigr) (T^ix_0)\, .$$
So we get, using the fact that the function $f_j\circ T$ is $\tau\,$-$\,$continuous on $K$,
\begin{eqnarray*}
 \int_\NN (\mathbf 1_{T(K)} f_j)(T^{i+1}x_0)\, d\frak m(i)&\geq &\int_\NN \bigl(\mathbf 1_{K}\cdot \left(f_j\circ T\right)\bigr)(T^{i}x_0)\, d\frak m(i)
 =\int_{K} (f_j\circ T)\, d\mu_{K}\, .
\end{eqnarray*}

Therefore, we obtain 
\begin{eqnarray*} \mu_{T(K)}(E)&\geq &\lim_{j\to\infty} \int_{K} \left(f_j\circ T\right)\, d\mu_{K}
=\int_{K} \bigl(\mathbf 1_{E\cap T(K)}\circ T\bigr)\, d\mu_{K}
.
\end{eqnarray*}

Since $1_{E\cap T(K)}\circ T\ge  \mathbf 1_{T^{-1}(E)\cap K}$, this yields that 
\begin{eqnarray*}\mu_{T(K)}(E)
&\geq&\int_{K} \mathbf 1_{T^{-1}(E)\cap K}\, d\mu_{K}
=\mu_{K} (T^{-1}(E))\, ,
\end{eqnarray*}
which proves our claim.
Since $T(K)\in\mathcal K_\tau$, it follows  that $\mu_K(T^{-1}(E))\leq \mu(E)$ for every $K\in\mathcal K_\tau$, and hence that $\mu(T^{-1}(E))\leq \mu(E)$. This concludes the proof of Fact \ref{fact5}.
\qed
\par\medskip
If we normalize the \mea\ $\mu$ by setting $m=\frac1{\mu(X)}\, \mu$, we have thus proved that $m$ is a $T$-$\,$\inv\ Borel \prob\ \mea\ on $X$ with full support, and this completes the proof of Theorem \ref{th2}.
\end{proof}

\begin{remark} What we have in fact proved is the following result. \emph{Let $(X,T)$ be a Polish dynamical system and let $\tau$ be any topology on $X$ which is coarser that the original topology but with the same Borel sets, whose compact sets are metrizable, and such that $T$ is continuous with respect to $\tau$. Then, for any invariant mean $\mathfrak m$ on $\ell^\infty(\NN)$ and any point $x_0\in X$, one can find a $T\,$-$\,$invariant finite Borel measure $\mu$ on $X$ such that $\mu(K)\geq \mathfrak m(\mathcal N_T(x_0,K))$ for every $\tau\,$-$\,$compact set $K\subset X$.} 
\end{remark}

\begin{remark}\label{rem1} The measure $\mu$ constructed above may not be a probability measure, so we do have to normalize it. Indeed, we have 
$$\mu(X)=\sup_{K\in\mathcal K_\tau}\mu_K(K)=\sup_{K\in\mathcal K_\tau} \int_\NN \mathbf 1_K (T^ix_0)\, d\mathfrak m(i)$$
and this may very well be smaller than $1$ if $\sup\limits_{K\in\mathcal K_\tau} \underline{\rm dens}\; \mathcal{N}_{T}(x_{0},K)<1\, .$
\end{remark}

\subsection{There are many invariant measures with full support} It is worth mentioning that as soon as there exists at least one $T$-$\,$invariant measure with full support, the set of all such measures is in fact a large subset of $\mathcal P_T(X)$ in the Baire Category sense. This (well-known) observation will be needed below.

\begin{lemma}\label{Gdelta}
Let $(X,T)$ be a Polish dynamical system, and denote by $\mathcal P_{T,*}(X)$ the set of all $T$-$\,$invariant Borel probability measures on $X$ with full support. If 
$\mathcal P_{T,*}(X)\neq\emptyset$, then $\mathcal P_{T,*}(X)$ is a dense $G_\delta$ subset of $\mathcal P_T(X)$.
\end{lemma}
\begin{proof} Let $(V_j)_{j\geq 1}$ be a countable basis of (nonempty) open sets for $X$. Then a measure $m\in\mathcal P(X)$ has full support if and only if $m(V_j)>0$ for all 
$j\ge 1$; and since the maps $\mu\mapsto \mu(V_j)$ are lower semi-continuous on 
$\mathcal P(X)$, it follows that $\mathcal P_{T,*}(X)$ is $G_\delta$ in $\mathcal P_T(X)$. Now, assume that $\mathcal P_{T,*}(X)\neq\emptyset$, and let us choose any element
$\mu_0$ of $ \mathcal P_{T,*}(X)$. If $m\in\mathcal P_T(X)$ is arbitrary, then the measure $m_\varepsilon:= (1-\varepsilon)\, m+ \varepsilon\,\mu_0$ is $T$-$\,$invariant for any 
$\varepsilon\in (0,1)$, and it has full support because $\mu_0$ has full support; in other words, $m_\varepsilon\in\mathcal P_{T,*}(X)$. Since $m_\varepsilon\to m$ as 
$\varepsilon\to 0$, this shows that $\mathcal P_{T,*}(X)$ is dense in $\mathcal P_T(X)$.
\end{proof}

\subsection{Ergodic measures with full support}
In view of Questions \ref{Q6} and \ref{Q10}, it is of course natural to wonder whether the above construction can give rise to an \emph{\erg\ } \mea\ with full support. There is no reason at all that the \mea\ $\mu$ constructed in the proof of Theorem \ref{th2} should be \erg. Still, if we were able to prove directly that $\mu(HC(T))=1$, then we would get for free an 
\erg\ \prob\ \mea\ with full support for $T$: indeed, as already explained in the introduction, it would follow directly from the \erg\ decomposition theorem that there exists an 
\erg\ \prob\ \mea\ $\nu $ such that $\nu(HC(T))>0$ (in fact $\nu(HC(T))=1$ by ergodicity), and such a \mea\ $\nu $ would necessarily have full support.
However, we see no reason either that the \mea\ $\mu $ constructed in the proof of Theorem \ref{th2} should satisfy $\mu(HC(T))=1$. The next proposition clarifies this a little bit.

\begin{proposition}\label{rem2}  Let $(X,T)$ be a Polish dynamical system. Assume that $X$ is endowed with a Hausdorff topology $\tau$ coarser than the original topology such that every point of $X$ has a neighbourhood basis (with respect to the original topology) consisting of $\tau\,$-$\,$compact sets, and that $T$ is a \emph{homeomorphism} of $(X,\tau)$. Then, the following assertions are equivalent.
\begin{enumerate}
\item[\rm (1)] $T$ admits an ergodic measure with full support;
\item[\rm (2)] there exists a point $x_0\in X$ such that 
$$\lim_{N\rightarrow\infty}\;\underline{\hbox{\rm dens}}\;\bigcup_{r=0}^{N} (\mathcal{N}_T(x_{0},V)-r)=1\quad \hbox{for\;every\;open\;set\;}V\neq \emptyset \,.$$

\end{enumerate}
Moreover, if a point $x_0\in X$ satisfies {\rm (2)}, then the measure $\mu$ constructed in the proof of Theorem \ref{th2} starting from $x_0$ satisfies $\mu(HC(T))=1$.

\end{proposition}
\begin{proof} Assume that $T$ admits an ergodic measure $\nu$ with full support. By the pointwise ergodic theorem and since the space $X$ is second-countable, 
$\nu\,$-$\,$almost every point $x_0\in X$ satisfies 
$$\liminf_{n\to\infty}\frac 1n \sum_{i=1}^n \mathbf 1_{U} (T^ix_0)\geq \nu (U)\quad {\rm for\;every\;open\;set\;}U\neq \emptyset\, . $$ 
Let us fix such a point $x_0$. Since the sum in the left-hand side is equal to the cardinality of the set $\mathcal N_T(x_0,U)\cap [1,n]$, we have $\underline{\rm dens}\;\mathcal N_T(x_0,U)\geq \nu(U)$ for every open set $U\neq \emptyset$. Applying this with $U=\bigcup_{r=0}^N T^{-r}(V)$ for a given open set $V\neq\emptyset$ and $N\in\NN$, and observing that 
$\mathcal N_T(x_0, T^{-r}(V))=\mathcal N_T(x_0,V)-r$ for every $r\in\{ 0,\dots ,N\}$, it follows that 
$$\underline{\hbox{\rm dens}}\;\bigcup_{r=0}^{N} (\mathcal{N}_T(x_{0},V)-r)\geq \nu\left(\bigcup_{r=0}^N T^{-r}(V)\right) 
\quad {\rm for\;every\;open\;set\;}V\neq \emptyset\, .$$
Since $\nu\left(\bigcup_{r=0}^\infty T^{-r}(V)\right)=1$ by ergodicity, this shows that (2) is satisfied.

\smallskip
To conclude the proof, it is now enough to show that if a point $x_0\in X$ satisfies (2), then the measure $\mu$ constructed in the proof of Theorem \ref{th2} starting from $x_0$ satisfies $\mu(HC(T))=1$. 

\smallskip
With the notation of the proof of Theorem \ref{th2}, we have $\bigcup_{r=0}^N T^{-r}(K)\in\mathcal K_\tau$ for any $K\in \mathcal{K}_{\tau}$ and every integer 
$N\ge 0$, because $T$ is assumed to be a homeomorphism of $(X,\tau)$. So we may write
\begin{eqnarray*}
\mu\left(\bigcup_{r=0}^{N}T^{-r}(K)\right)&\ge&
\mu_{\bigcup_{r=0}^{N}T^{-r}(K)}\left(\bigcup_{r=0}^{N}T^{-r}(K)\right)\\
&\ge&
\liminf_{n\to\infty} \frac1n\sum_{i=1}^n \mathbf 1_{\bigcup_{r=0}^{N}T^{-r}(K)}(T^{i}x_{0})\\
&=&\underline{\textrm{dens}}\left\{i\in\N \textrm{ ; } T^{i}x_{0}\in
\bigcup_{r=0}^{N}T^{-r}(K)\right\}\\
&=&\underline{\textrm{dens}}\,\bigcup_{r=0}^{N} (\mathcal{N}_T(x_{0},K)-r)\, .
\end{eqnarray*}

By (2), it follows that $\mu\left(\bigcup_{r=0}^\infty T^{-r}(K)\right)=1$ whenever $K\in\mathcal K_\tau$ has nonempty $\tau_X$-$\,$interior; and since every nonempty open set contains such a set $K$, this means that $\mu\left(\bigcup_{r=0}^\infty T^{-r}(V)\right)=1$ for every open set $V\neq\emptyset$. Since $HC(T)=\bigcap_{j\geq 1} \bigcup_{r=0}^\infty T^{-r}(V_j)$, where $(V_j)_{j\geq 1}$ is a countable basis of (nonempty) open sets for $X$, we conclude that $\mu(HC(T))=1$.
\end{proof}

\begin{remark}
Assuming only the existence of one frequently hypercyclic point $x_0$, we see no a priori reason for condition (2) above to hold true. For example, it is not difficult to construct subsets $D$ of $\N$ with $\underline{\textrm{dens}}(D)>0$ such that  $$\sup_{N\ge 0}\;\underline{\textrm{dens}}\,\bigcup_{r=0}^{N} (D-r)<1.$$
On the other hand, if $D$ is assumed to have  \emph{positive Banach density} (i.e. if its upper Banach density and lower Banach density are equal and positive), then it is a result of Hindman \cite{H} that 
$$\lim_{N\to\infty}\,\underline{\rm dens}\;\bigcup_{r=0}^{N} (D-r)=1.$$
So if assumption (i) in Theorem \ref{th2} is replaced with a much stronger one, namely
\begin{itemize}
\item[\rm (i')] \emph{there exists a point $x_{0}$ in $X$ such that for every non empty $\tau_{X}$-open set $V\subset X$, the set $\mathcal{N}_T(x_{0},V)$ has positive Banach density,}
\end{itemize}
then $T$ admits an \erg\ \prob\ \mea\ with full support by Proposition \ref{rem2}. But assumption (i') does not seem to be realistic at all, because if we stay at the abstract 
measure-theoretic level, the a priori existence of an \erg\ \mea\  with full support does not imply it (see \mbox{e.g.} \cite{JR}).
\end{remark}

\subsection{Frequent recurrence}

Recall that a Polish dynamical system $(X,T)$ is said to be \emph{recurrent} if, for every nonempty open set $V\subset X$, one can find $n\in\NN$ such that $T^n(V)\cap V\neq \emptyset$. Recurrence is of course a central theme in both topological dynamics and ergodic theory; see \mbox{e.g} Furstenberg's book \cite{Fu}. In the linear setting, it has been more particularly studied recently in \cite{CMP}.

\smallskip
Let us say that a Polish dynamical system $(X,T)$ is \emph{frequently recurrent} if, for every nonempty open set $V\subset X$, one can find a point $x_V\in V$ such that 
$\mathcal N_T(x_V, V)$ has positive lower density. Of course, this does not imply frequent hypercyclicity of the system: for example, $(X, id_X)$ is frequently recurrent. However, this notion allows us to characterize, among those Polish dynamical systems considered in Theorem \ref{th2}, the ones that admit an invariant measure with full support.

\begin{proposition}\label{charac} Let $(X,T)$ be a Polish dynamical system. Assume that $X$ is endowed with a Hausdorff topology $\tau$ coarser than the original topology such that every point of $X$ has a neighbourhood basis (with respect to the original topology) consisting of $\tau\,$-$\,$compact sets, and that $T$ is continuous with respect to the topology $\tau$. Then the following assertions are equivalent:

\begin{itemize}
\item[\rm (1)] $T$ admits an invariant measure with full support;
\item[\rm (2)] for each open set $V\neq\emptyset$, there is an ergodic measure $\mu_V$ for $T$ such that $\mu_V(V)>0$;
\item[\rm (3)] $T$ is frequently recurrent.
\end{itemize}
\end{proposition}
\begin{proof} The implication $(1)\implies (2)$ follows from the ergodic decomposition theorem, and $(2)\implies (3)$ is a direct consequence of the pointwise ergodic theorem. Finally, assume that $T$ is frequently recurrent. Then, the proof of Theorem \ref{th2} shows that for each nonempty open set $V\subset X$, one can find a $T$-$\,$invariant measure $m_V$ such that $m_V(V)>0$: just carry out the construction starting from the point $x_0=x_V$ given by the frequent recurrence assumption. If we now choose a countable basis of (nonempty) open sets $(V_j)_{j\geq 1}$ for $X$, then $m:=\sum_1^\infty 2^{-j} m_{V_j}$ is an invariant measure with full support.
\end{proof}

\begin{remark} An examination of the proof reveals that (3) can be replaced by a formally weaker assumption, namely that for each open set $V\neq\emptyset$, one can find a point $x_V\in X$ such that $\mathcal N_T(x_V,V)$ has positive \emph{upper} density. Indeed, one just has to note that given any $\phi\in\ell^\infty(\NN)$ and $x_0\in X$, the construction used in the proof of Theorem \ref{th2} can be made with an invariant mean $\mathfrak m$ satisfying 
$\int_\NN \mathbf \phi(i)\, d\mathfrak m(i)\geq \limsup_{n\to\infty} \,\frac 1n\sum_{i=1}^n \phi(i)$ for this particular $\phi\in\ell^\infty(\NN)$. If we start with a nonempty open set $V$ and take $\phi(i):=\mathbf 1_K(T^ix_{V'})$, where $K$ is a $\tau\,$-$\,$compact set contained in $V$ with nonempty interior $V'$, this produces an invariant measure $m_V$ such that $m(V)>0$; so one can repeat the proof of the implication (3)$\implies$(1).
\end{remark}
\subsection{Continuous invariant \mea s with full support}

One might wonder under which conditions a Polish dynamical system $(X,T)$ 
satisfying the assumptions of Theorem \ref{th2} admits a \emph{continuous} \inv\ \prob\ \mea\ $m$ with full support (i.e. an invariant \mea\ $m$ such that $m(\{a\})=0$ for every $a\in X$). It turns out that it is quite often possible to deduce the existence of such a \mea\ directly from the existence of an \inv\ \mea\ with full support. In particular, we will see in Section \ref{sec3} that it is always the case for linear dynamical systems. We first observe 

\begin{fact}\label{obvious} Let $T:X\to X$ be a continuous self-map of a Hausdorff topological space $X$, and let $m$ be a $T$-$\,$invariant Borel \prob\ measure on $X$. If 
$a\in X$ is such that $m(\{ a\})>0$, then $a$ is a periodic point of $T$. In this case, $m(\{T^{k} a\})=m(\{ a\})$ for every $k\geq 0$.
\end{fact}

\begin{proof} If $a$ is not periodic for $T$ then the sets $T^{-n}(\{ a\})$, $n\geq 0$, are pairwise disjoint. Since all these sets have measure $m(\{ a\})$ and $m$ is a \prob\ \mea, it follows that $m(\{ a\})=0$, which is a contradiction. Suppose now that $a$ is a periodic point for $T$ with period $N\ge 1$. Since $T^{N-1}a$ belongs to 
$T^{-1}(\{a\})$, we have $m(\{T^{N-1}a\})\le m(\{ a\})$. In the same way, $m(\{T^{N-2}a\})\le m(\{ T^{N-1}a\})$, etc... so that we finally obtain that 
$m(\{a\})\le m(\{ Ta\})\le\ldots\le m(\{T^{N-1}a\})\le m(\{ a\})$. Hence, all the quantities $m(\{T^{k}a\})$, $0\le k<N$, are equal.
\end{proof}

\begin{remark}\label{bi-obvious} It follows in particular from Fact \ref{obvious} that the finitely supported $T$-$\,$invariant probability measures on $X$ are exactly the convex combinations of periodic measures for $T$.
\end{remark}

We are now ready to prove

\begin{proposition}\label{propcontinuous} Let $(X,T)$ be a Polish dynamical system, where the space $X$ is assumed to have no isolated point. Assume that there exists a finite set $F_0\subset X$ such that, for every $N\in\NN$, the set $\{ x\in X;\; x\not\in F_0\;{\rm and}\; T^Nx=x\}$ has no isolated point. Then, if $T$ admits an invariant  measure with full support, it also admits one which is continuous.
\end{proposition}

\begin{proof} Let $\mu$ be an invariant \prob\ measure for $T$ with full support. Denote by $\mu_c$ and $\mu_d$ respectively the continuous  part and  the discrete part of $\mu$. If we set $D=\{ a\in X;\; \mu(\{ a\}) >0\}$, then $D$ is a countable set and $\mu_d$ is supported on $D$. By Fact \ref{obvious} above, each point $a\in D$ is 
$T$-$\,$periodic and we may write
$$\mu_d=\sum_{a\in D} c_a \nu_a,$$ where $c_a> 0$ and $\nu _{a}$ is a periodic \mea\ supported on the orbit of the periodic vector  $a$. 
The measure $\widetilde\mu:=\mu-\sum_{a\in D\cap F_0} c_a\nu_a$ is $T$-$\,$invariant, and it also has full support 
because $D\cap F_0$ is finite and $X$ has no isolated point. So by replacing $\mu$ with $\widetilde\mu$, we may in fact assume that $D\cap F_0=\emptyset$.

Being a sum of (multiples of) periodic measures,  the measure $\mu_d$ is $T$-$\,$invariant, so $\mu_c$ is $T$-$\,$invariant as well. Therefore, in order to prove Proposition \ref{propcontinuous} it is enough to show that for every point $a\in D$, one can find a continuous, $T$-$\,$invariant measure 
$m_a$ whose support contains $a$. Indeed, in this case the \mea\ 
$m$ defined by $m:=\mu_c +\sum_{a\in D} \varepsilon_a m_a$, where the $\varepsilon_a$ are small enough positive coefficients, will be a finite invariant measure for $T$ with full support.

So let us fix $a\in D$ and $N\geq 1$ such that $T^N a=a$. Since $a\not\in F_0$ and $F_0$ is closed in $X$, we can choose a decreasing countable basis $(V_j)_{j\geq 1}$ of open neighbourhoods of $a$ such that $V_j\cap F_0=\emptyset$ for all $j\ge 1$. Then for each $j\geq 1$, the set $C_j=\{ x\in V_j;\; T^N x=x\}$ is nonempty (because it contains the point $a$) and has no isolated point. Moreover, since $C_j$ is defined as the intersection of an open set and a closed set in the Polish space $X$, it is a Polish space as well. Therefore, each set $C_j$ contains a compact set $K_j$ which is homeomorphic to the Cantor space $\{ 0,1\}^\NN$. 

Let us choose for each $j\geq 1$ a continuous probability measure $m_j$ on $K_j$ whose support is the set $K_j$, and consider $m_j$ as a Borel measure on $X$. Now, define the \mea\ $m_{a}$ as
$$m_a:=\sum_{j=1}^\infty 2^{-j} \left( \frac 1N\sum_{n=0}^{N-1} m_j\circ T^{-n}\right) . $$
This is a probability measure, and $m_a$ is easily seen to be $T$-$\,$invariant because $T^Nx \equiv x$ on $K_j$ for all $j\geq 1$. Finally, the support of the measure $m_a$ contains all compact sets $K_j$, and since the sets $K_j$ accumulate to $\{ a\}$, it follows that the support of $m_a$ contains the point $a$. This finishes the proof of Proposition 
\ref{propcontinuous}.
\end{proof}

\section{Invariant measures for linear systems}\label{sec3}

In this section, we prove Theorem \ref{th1}. The first part follows immediately from Theorem \ref{th2}, but our proof of the second part relies on some extra and specifically linear arguments.

\subsection{Continuous invariant measures with full support}

If $X$ is a reflexive separable Banach space with norm topology $\tau_X$, one can apply Theorem \ref{th2} by taking 
as $\tau$ the weak topology of $X$. Since $X$ is reflexive, all closed balls are $\tau\,$-$\,$compact and hence each point $x\in X$ has a neighbourhood basis consisting of 
$\tau\,$-$\,$compact sets. If $T:X\to X$ is a bounded linear operator on $X$, then $T$ is continuous with respect to the weak topology; that is, assumption (ii) in Theorem \ref{th2} is satisfied. Thus, we immediately get that if $T$ is a frequently hypercyclic operator on $X$, then $T$ admits an invariant measure with full support. 

\smallskip
Moreover, without assuming that $X$ is reflexive, one can deduce immediately from Proposition \ref{propcontinuous} that if an operator $T\in\mathfrak L(X)$ admits an invariant measure with full support, then it also admits one which is continuous: just take $F_0=\{ 0\}$ in Proposition \ref{propcontinuous}, and observe that for every $N\ge 1$, the set 
$\{ x\in X;\; x\not=0 \;{\rm and}\; T^Nx=x\}$ is either empty or a nonempty open set in a closed linear subspace of $X$ of dimension 
 at least $1$.

\smallskip
So, if $X$ is reflexive and $T$ is frequently hypercyclic, then $T$ admits at least one continuous invariant measure with full support. In fact, one can say a little bit more:

\begin{proposition}\label{continuous2} Let $T\in\mathfrak L(X)$, where $X$ is a Polish topological vector space. If $T$ admits an invariant measure with full support, then the continuous, $T$-$\,$invariant \prob\ measures with full support form a dense $G_\delta$ subset of $\mathcal P_T(X)$. 
 This holds in particular if $X$ is reflexive and $T$ is frequently hypercyclic.
 \end{proposition}
\begin{proof}[Proof of Proposition \ref{continuous2}] We first show that the set $\mathcal P^c(X)$ of all continuous probability measures on $X$ is $G_\delta$ in $\mathcal P(X)$. This is in fact true for any Polish space $X$:

\begin{fact}\label{Gdelta2} If $X$ is a Polish space, then 
$\mathcal P^c(X)$ is a $G_\delta$ subset of $\mathcal P(X)$.
\end{fact} 
\begin{proof}
Le $\widehat X$ be a metrizable compactification of $X$. Then any measure $\mu\in\mathcal P(X)$ can be identified in a canonical way with a measure 
$\widehat\mu\in\mathcal P(\widehat X)$, namely the measure defined by $\widehat\mu (A):=\mu(A\cap X)$ for every Borel set $A\subset\widehat X$. The map $\mu\mapsto\widehat\mu$ is continuous from $\mathcal P(X)$ into $\mathcal P(\widehat X)$, and a measure $\mu\in\mathcal P(X)$ is continuous if and only if $\widehat\mu$ is. So it is enough to show only that $\mathcal P^c(\widehat X)$ is $G_\delta$ in $\mathcal P(\widehat X)$. In other words, we may assume from the beginning that $X$ is \emph{compact}.

\smallskip
Having fixed a compatible metric for $X$, we can find for every 
$n\in\NN$ a finite covering $(V_{n,i})_{i\in I_n}$ of $X$ by open sets $V_{n,i}$ with diameter less than $2^{-n}$.
Then it is easy to check that a measure $\mu\in\mathcal P(X)$ belongs to $\mathcal P^{c}(X)$ if and only if 
$$\forall k\in\NN\;\Bigl(  \exists n\;\forall i\in I_n\;:\;\mu(\overline V_{n,i})<\frac1k\Bigr) . $$
For every fixed $k\in\NN$, the condition under brackets defines an open subset of $\mathcal P(X)$, because the sets $\overline V_{n,i}$ are closed in $X$. So the above formula shows that $\mathcal P^{c}(X)$ is $G_\delta$.
\end{proof}

\smallskip
Now, assume that $T$ admits an invariant measure with full support, and let us denote by $\mathcal P_{T,*}^c(X)$ the family of all continuous, $T$-$\,$invariant \prob\ measures with full support (the star symbol is here to remind that measures in $\mathcal P_{T,*}^c(X)$ are required to have full support). By Facts \ref{Gdelta2} and \ref{Gdelta}, 
$\mathcal P_{T,*}^c(X)=\mathcal P_{T,*}(X)\cap\mathcal P_{T}^c(X)$ is a $G_\delta$ subset of $\mathcal P_T(X)$; so we just have to show that 
$\mathcal P_{T,*}^c(X)$ is dense in $\mathcal P_T(X)$. In fact, by the Baire category theorem it would be enough to show that $\mathcal P^c_T(X)$ is dense in $\mathcal P_T(X)$, but it is not harder to prove directly that $\mathcal P_{T,*}^c(X)$ is dense. Note that $\mathcal P_{T,*}^c(X)$ is \emph{nonempty} by Proposition \ref{propcontinuous}.

\smallskip
We first show that one can at least approximate the Dirac mass 
$\delta_0$ by measures in $\mathcal P_{T,*}^c(X)$.
\begin{fact}\label{Dirac0} The Dirac mass $\delta_0$ belongs to the closure of $\mathcal P_{T,*}^c(X)$ in $\mathcal P_T(X)$.
\end{fact}
\begin{proof}[Proof of Fact \ref{Dirac0}] We start with the following
 
\begin{claim}\label{claim0} For every $\varepsilon\in (0,1)$ and every neighbourhood $W$ of $0$ in $X$, there exists a measure $\nu\in \mathcal P_{T,*}(X)$ such that 
$\nu(W)>1-\varepsilon.$
\end{claim}

\begin{proof}[Proof of Claim \ref{claim0}]  Let $m$ be any continuous $T$-$\,$invariant probability measure with full support. For any 
$\eta>0$, consider the ``dilated" measure $m^\eta$ defined 
by setting $m^\eta(A):=m\left(\frac1\eta\cdot A\right)$ for any Borel set $A\subset X$. This measure is still continuous, it is 
$T$-$\,$invariant by the linearity of $T$, and it has full support. 
Choose a compact set $K\subset X$ such that $m(K)>1-\varepsilon$, and then $\eta>0$ such that $K\subset \frac 1\eta W$. Then $m^\eta(W) >1-\varepsilon$, so that the measure 
$\nu=m^\eta$ satisfies the conclusion of the Claim.
\end{proof}

The deduction of the above fact from Claim \ref{claim0} is standard, but we give the details for convenience of the reader. Let $(W_k)_{k\ge 1}$ be a decreasing countable basis of open neighbourhoods of $0$ in $X$, and let $(\varepsilon_k)_{k\ge 1}$ be a sequence of positive numbers tending to $0$ as $k$ tends to infinity. For each $k\ge 1$, one can apply Claim \ref{claim0} to get a measure 
$\nu_k\in \mathcal P_{T,*}^c(X)$ such that $\nu_k(W_k)>1-\varepsilon_k$. If $f$ is any bounded, real-valued continuous function on $X$, then 
$$\int_X f \, d\nu_k=\int_{W_k} f\, d\nu_k+\int_{X\setminus W_k} f\, d\nu_k\, .$$
The second term on the right-hand side clearly tends to $0$ as $k\to\infty$, whereas 
$$\frac1{\nu_k(W_k)} \int_{W_k} f\, d\nu_k -f(0)=\frac1{\nu_k(W_k)}\int_{W_k} \bigl(f-f(0)\bigr)\, d\nu_k\To 0\quad \textrm{ as } k\to\infty$$
by the continuity of $f$ at $0$ (and the fact that $\nu_k(W_k)$ is bounded below). Since $\nu_k(W_k)\to 1$, it follows that $\int_Xf\, d\nu_k\to f(0)$ as $k\to\infty$, for any 
$f\in\mathcal C_b(X)$. In other words, $\nu_k\to \delta_0$ in $\mathcal P(X)$. 
This finishes the proof of Fact \ref{Dirac0}.
\end{proof} 

\par\smallskip
Proposition \ref{continuous2} can now be proved by combining the above fact and a simple \emph{convolution} argument.

\smallskip
Let us fix $m\in\mathcal P_T(X)$. We want to show that $m$ belongs to the closure of $\mathcal P_{T,*}^c(X)$ in $\mathcal P(X)$; in other words, that one can find a sequence 
$(\mu_k)_{k\ge 1}$ of elements of $\mathcal P_{T,*}^c(X)$ such that $\mu_k\to m$.

By Fact \ref{Dirac0}, there exists a sequence $(\nu_k)_{k\ge 1}$ of elements of $\mathcal P_{T,*}^c(X)$ such that $\nu_k\to\delta_0$. Then set $\mu_k:=\nu_k*m$, the convolution product of $\nu_k$ and $m$. For any bounded Borel function $f:X\to \RR$, we have by definition
$$ \int_X f\, d\mu_k=\int_{X\times X} f(x+y)\, d\nu_k(x)\, dm(y)\, .$$
\par\smallskip
Since $\nu_k\to\delta_0$, it is easily checked that $\mu_k\to m$ in $\mathcal P(X)$. Indeed, for any $f\in\mathcal C_b(X)$, the function $f*m$ defined by 
$(f*m)(x)= \int_X f(x+y) \, dm(y)$ belongs to $\mathcal C_b(X)$, so that
$$\int_X f\, d\mu_k=\int_X (f*m)\, d\nu_k\To (f*m)(0)=\int_X f\, dm.$$
In order to conclude the proof, it remains to check that each \mea\ $\mu_k$ belongs to 
$\mathcal P_{T,*}^c(X)$.
\par\smallskip
The $T$-$\,$invariance of $\mu_k$ follows from the linearity of $T$ and the $T$-$\,$invariance of $m$ and $\nu_k$. Indeed, for each bounded Borel function $f:X\to\RR$ we have
\begin{eqnarray*}
 \int_X (f\circ T)\, d\mu_k&=&\int_{X\times X} f(Tx+Ty)\, d\nu_k(x)\, dm(y)
 =\int_{X\times X} f(Tx+y)\, d\nu_k(x)\, dm(y)\\
 &=&\int_{X\times X} f(x+y)\, d\nu_k(x)\, dm(y)
 =\int_X f\, d\mu_k\, .
\end{eqnarray*}

The measure $\mu_k$ has full support because $\nu_k$ does: if $U$ is a nonempty open set in $X$, then $\nu_k(U-y)>0$ for all $y\in X$ and hence
$$\mu_k(U)=\int_X \nu_k(U-y)\, dm(y)>0\, .$$

Finally, the measures $\mu_k$ are continuous because the $\nu_k$ are: for any point $a\in X$ we have $\mu_k(\{ a\})=\int_X \nu_k(\{ a-y\})\, dm(y)=0$. This finishes the proof of Proposition \ref{continuous2}.
\end{proof}

\begin{remark}\label{dual} The assumption of the ``in particular" part of Proposition \ref{continuous2} can be relaxed: it is enough to assume that the separable Banach space $X$ is a dual space, and that the operator $T:X\to X$ is an adjoint operator. Indeed,  Theorem \ref{th1} applies in the same way if we take as $\tau$ the $w^*$ topology of the dual space $X$, and all the remaining arguments are unchanged.
\end{remark}

\begin{remark}\label{rem5} There is an abstract statement lying behind the above proof, which reads as follows: \emph{if $\mathcal M$ is a nonempty subset of $\mathcal P_T(X)$ which is stable under dilations and such that $\mathcal P_T(X)*\mathcal M\subset \mathcal M$, then $\mathcal M$ is dense in $\mathcal P_T(X)$.}
\end{remark}

\subsection{Invariant measures with $m({\rm Per}(T))=0$} 

In this sub-section, our aim is to show that if $T$ is a frequently hypercyclic operator on a reflexive Banach space $X$, then one can find a 
$T$-$\,$invariant Borel probability measure $m$ with full support such that $m({\rm Per}(T))=0$, where 
 ${\rm Per}(T)$ is the set of all periodic points of $T$. Such a \mea\ is necessarily continuous by Fact \ref{obvious}, but the requirement that $m({\rm Per}(T))=0$ is of 
 course much stronger. 
 
 \smallskip
In fact, we will use neither the frequent hypercyclicity of $T$, nor the reflexivity of $X$, but only the fact that $T$ admits an invariant measure with full support. The result we shall prove is a strengthening of Proposition \ref{continuous2}. Here and afterwards, we use the following notation: for any Borel set $A\subset X$, we set 
$\mathcal P_{T,*}(A):=\mathcal P_{T,*}(X)\cap \mathcal P(A)$. In words, $\mathcal P_{T,*}(A)$ is the family of all $T$-$\,$invariant probability measures $m$ on $X$ with full support such that $m(A)=1$.

\begin{proposition}\label{residual2}
If $T$ is a continuous linear operator on a Polish topological vector space $X$ such that $\mathcal P_{T,*}(X)\neq\emptyset$ and $T^N\neq Id$ for all $N\geq 1$, then 
$\mathcal P_{T,*}(X\setminus{\rm Per}(T))$ is a dense $G_\delta$ subset of $\mathcal P_T(X)$.
\end{proposition}

\smallskip
The final part of Theorem \ref{th1} follows at once from this result. In fact, Proposition \ref{residual2} even allows us to characterize the operators $T$ acting on a reflexive space $X$ for which
$\mathcal P_{T,*}(X\setminus{\rm Per}(T))\neq\emptyset$:

\begin{corollary}\label{charac2}
Let $T$ be a bounded operator acting on a reflexive separable Banach space $X$ or, more generally, an adjoint operator acting on a separable dual space. Then $T$ admits an invariant measure $m$ with full support such that $m({\rm Per}(T))=0$ if and only if $T$ is frequently recurrent and $T^N\neq Id$ for every $N\geq 1$.
\end{corollary}
\begin{proof} Combine Propositions \ref{charac} and \ref{residual2}.
\end{proof}

\smallskip
Proposition \ref{residual2} will follow easily from the next Lemma. This Lemma is a little bit more than what is really needed for our purpose, but the generality might be useful elsewhere. In what follows, we denote by $\ker^*(T)$ 
the \emph{generalized kernel} of the operator $T$, \mbox{i.e.} 
$$\ker^*(T)=\bigcup_{k\in\NN} \ker (T^k)\, .$$

Recall also that a set $A\subset X$ is said to be \emph{dilation-invariant} if $r\cdot A=A$ for every $r>0$.

\begin{lemma}\label{residual} 
Let $T$ be a continuous linear operator on $X$ such that $\mathcal P_{T,*}(X)\neq\emptyset$, and let $F$ be a closed subset of $X$ such that
\begin{itemize}
\item[\rm{(i)}] $F-F$ is dilation-invariant and nowhere dense in $X$;
 
\item[\rm{(ii)}] $T(F\setminus \ker^*(T))\subset F$;
 
\item[\rm{(iii)}] $T\bigl(\,\overline{F-F}\,\setminus\ker^*(T)\bigr)\subset \overline{F-F}$.
\end{itemize}
Then,
$\mathcal P_{T,*}(X\setminus F)$ is a dense $G_\delta$ subset of $\mathcal P_T(X)$.
\end{lemma}

We shall in fact use this result only through the following immediate consequence.
\begin{corollary}\label{residualcoro} If $\mathcal P_{T,*}(X)\neq \emptyset$ and if $F$ is any proper closed linear subspace of $X$ such that $T(F)\subset F$, then  
$\mathcal P_{T,*}(X\setminus F)$ is a dense $G_\delta$ subset of $\mathcal P_T(X)$.
\end{corollary}

\par\smallskip
Taking this result for granted, it is easy to prove Proposition \ref{residual2} and hence to conclude the proof of Theorem \ref{th1}.

\begin{proof}[Proofs of Proposition \ref{residual2} and Theorem \ref{th1}] Let $T$ be such that 
$\mathcal P_{T,*}(X)\neq\emptyset$ and $T^N\neq Id$ for all $N\geq 1$. The set ${\rm Per}(T)$ of all periodic points of $T$ can be written as ${\rm Per}(T)=\bigcup_{N\ge 1} F_N$, where $F_N=\ker(T^N-I)$. Each $F_N$ is a proper closed subspace of $X$, and it is clear that $T(F_N)\subset F_N$. By Corollary \ref{residualcoro}, $\mathcal P_{T,*}(X\setminus F_N)$ is  a dense $G_\delta$ subset of $\mathcal P_T(X)$ for each $N\geq 1$. By the Baire Category theorem, it follows that $\bigcap_{N\ge 1}  \mathcal P_{T,*}(X\setminus F_N)$ is nonempty. This means exactly that one can find a measure $m\in\mathcal P_T(X)$ with full support such that 
$m({\rm Per}(T))=0$, which concludes the proof of Proposition \ref{residual2} (and hence of Theorem \ref{th1}).
\end{proof}

Before starting the proof of Lemma \ref{residual}, we first state three simple facts. 

\begin{fact}\label{P(O)}
If $O$ is an open subset of the Polish space $X$, then $\mathcal P_T(O)$ is a $G_\delta$ subset of $\mathcal P_T(X)$.
\end{fact}
\begin{proof}
A measure $m\in\mathcal P(X)$ is supported on 
$O$ if and only if $m( O)>1-2^{-k}$ for all $k\ge 1$. Since $O$ is open, the map $m\mapsto \mu(O)$ is lower semi-continuous on 
$\mathcal P(X)$. It follows that $\mathcal P(O)$ is a $G_\delta$ subset of $\mathcal P(X)$, and hence that $\mathcal P_T(O)$ is a
$G_\delta$ subset of $\mathcal P_T(X)$. 
\end{proof}

\begin{fact}\label{complementaire} Let $(X,T)$ be a Polish dynamical system, and let $m\in\mathcal P_T(X)$. If $E\subset X$ is  a Borel set such that either $T(E)\subset E$ or 
$T^{-1}(E)\subset E$, then the measure $\mathbf 1_{E}\, m$ is $T$-$\,$invariant.
\end{fact}
\begin{proof} Set $\mu:=\mathbf 1_E\, m$, and assume for example that $T(E)\subset E$, \mbox{i.e.} $E\subset T^{-1}(E)$. For any Borel set $A\subset X$, we have
\begin{eqnarray*} \mu\bigl(T^{-1}(A)\bigr)=\mathbf 1_E\, m\bigl(T^{-1}(A)\bigr)
\leq\mathbf 1_{T^{-1}(E)}\,m\bigl(T^{-1}(A)\bigr)
=(m\circ T^{-1})(E\cap A).
\end{eqnarray*}
Since $m$ is $T$-$\,$invariant, this means that $\mu\bigl(T^{-1}(A)\bigr)\leq \mu(A)$ for every Borel set $A$. Applying this with $X\setminus A$ in place of $A$, it follows 
that in fact $\mu(T^{-1}(A))= \mu(A)$ for every Borel set $A$, and hence that $\mu$ is $T$-$\,$invariant. If $T^{-1}(E)\subset E$, we obtain in the same way that  
$\mu(A)\geq \mu( T^{-1} (A))$ for every Borel set $A$, from which it follows that $\mu$ is $T$-$\,$invariant.
\end{proof}

\begin{fact}\label{variante} Let $T$ be a continuous linear operator on $X$ such that $\mathcal P_{T,*}(X)\neq\emptyset$. If $O\subset X$ is a nonempty open set such that 
$T^{-1}(O)\setminus \ker^*(T)\subset O$, then $\mathcal P_T(O)\neq\emptyset$.
\end{fact}
\begin{proof} By the discussion of the previous sub-section, we know that $T$ admits a continuous invariant measure $m$ with full support. Then $m(O)>0$, and multiplying $m$ by a suitable constant we may assume that $m(O)=1$. Set 
$\mu:= \mathbf 1_O\, m$. Since $m$ is continuous we have $m(\ker^*(T))=m\bigl( \bigcup_{k\in\NN} T^{-k}(\{ 0\})\bigr)=0$, so that in fact 
$\mu=\mathbf 1_{O\setminus\ker^*(T)}\, m$. Since $T^{-1}(\ker^*(T))=\ker^*(T)$, we have
$T^{-1}(O\setminus \ker^*(T))=T^{-1}(O)\setminus\ker^*(T)\subset O\setminus \ker^*(T)$, so the measure $\mu$ is $T$-$\,$invariant by Fact \ref{complementaire} above. Hence
$\mu$ belongs to $\mathcal P_T(O)$.
\end{proof}

\begin{proof}[Proof of Lemma \ref{residual}] By Fact \ref{P(O)} and Lemma \ref{Gdelta}, $\mathcal P_{T,*}(X\setminus F)=\mathcal P_T(X\setminus F)\cap \mathcal P_{T,*}(X)$ is a $G_\delta$ subset of $\mathcal P_T(X)$. The main point is to show that if $F$ satisfies the assumptions of Lemma \ref{residual}, then $\mathcal P_{T,*}(X\setminus F)$ is dense in 
$\mathcal P_T(X)$. Moreover, since $\mathcal P_{T,*}(X)$ is a dense $G_\delta$ subset of $\mathcal P_T(X)$ by Lemma \ref{Gdelta} again, it is in fact enough to show that the $G_\delta$ set $\mathcal P_T(X\setminus F)$ is dense in $\mathcal P_T(X)$. 
The proof will be quite similar to that of Proposition \ref{continuous2}. 

\smallskip
\begin{fact}\label{factDirac} The Dirac mass $\delta_0$ belongs to the closure of $\mathcal P_{T}(X\setminus \overline{F-F}\,)$ in $\mathcal P_T(X)$.
\end{fact}

\begin{proof}[Proof of Fact \ref{factDirac}] Exactly as for the proof of Fact \ref{Dirac0} above, this will follow from the next
 
\begin{claim}\label{claim} For every $\varepsilon\in (0,1)$ and every neighbourhood $W$ of $0$ in $X$, there exists a measure 
$\nu\in \mathcal P_{T}(X\setminus \overline{F-F}\,)$ such that 
$\nu(W)>1-\varepsilon.$
\end{claim}

\begin{proof}[Proof of Claim \ref{claim}]  Set $O:=X\setminus \overline{F-F}$. Then $O$ is a dense open subet of $X$, and since 
$T\bigl(\,\overline{F-F}\,\setminus\ker^*(T)\bigr)\subset \overline{F-F}$ we have $T^{-1}(O)\setminus\ker^*(T)\subset O$. Moreover, $O$ is dilation-invariant because $\overline{F-F}$ is. By Lemma \ref{variante}, one can find a $T$-$\,$invariant probability measure $\mu$ such that $\mu(O)=1$. For any 
$\eta>0$, the dilated measure $\mu^\eta$ defined as in Fact \ref{Dirac0} above is $T$-$\,$invariant (by the linearity of $T$) and still satisfies 
$\mu^\eta(O)=1$ (by the dilation-invariance of $O$). The same argument as in the proof of
Fact \ref{Dirac0} shows that if $\eta$ is sufficiently small, the measure 
$\nu:=\mu^\eta$ has the required properties.
\end{proof}
This proves Fact \ref{factDirac}.
\end{proof} 

Having established Fact \ref{factDirac}, we can now conclude the proof of Lemma \ref{residual} by the same kind of convolution argument as in the proof of Proposition \ref{continuous2}.
\par\smallskip
By Proposition \ref{continuous2} (or more accurately, Remark \ref{rem5}), it is enough to show that any \emph{continuous} $T$-$\,$invariant measure belongs to the closure of $\mathcal P_{T}(X\setminus F)$ in 
$\mathcal P(X)$. So let us fix a continuous measure $m\in\mathcal P_T(X)$. We have to find a sequence $(\mu_k)_{k\ge 1}$ of elements of $\mathcal P_{T}(X\setminus F)$ such that $\mu_k\to m$.
\par\smallskip
First observe that since $T(F\setminus \ker^*(T))\subset F$ and $T^{-1}(\ker^*(T))=\ker^*(T)$, we have $T(F\setminus\ker^*(T))\subset F\setminus\ker^*(T)$ and 
$T^{-1}\bigl((X\setminus F)\setminus \ker^*(T)\bigr)\subset (X\setminus F)\setminus \ker^*(T)$. Moreover, since $m$ is continuous and $T$-$\,$invariant, we also have 
$m(\ker^*(T))=m\bigl(\bigcup_{k\in\NN} T^{-k}(\{ 0\})\bigr)=0$. By Fact \ref{complementaire}, it follows that we can decompose $m$ as a convex combination of two measures $m_1,m_2\in\mathcal P_T(X)$ with $m_1(F)=0=m_2(X\setminus F)$: just set 
$$m_1:= \frac1{\scriptstyle m\left((X\setminus F)\setminus\ker^*(T)\right)} \,\mathbf 1_{\scriptstyle (X\setminus F)\setminus\ker^*(T)}\, m \quad \textrm{ and } 
\quad m_2:=\frac1{\scriptstyle m \left(F\setminus\ker^*(T)\right)} \,\mathbf 1_{\scriptstyle F\setminus\ker^*(T)}\, m $$ 
if  $m(F) m(X\setminus F)\neq0$, and $m_1=m$ or $m_2=m$ if $m(F)=0$ or $m(X\setminus F)=0$.  
\par\smallskip
Now $\mathcal P_{T}(X\setminus F)$ is a convex subset of 
$\mathcal P(X)$, and  the measure $m_1$ already belongs to $\mathcal P_T(X\setminus F)$.  So it is enough to approximate the measure $m_2$ by elements of $\mathcal P_T(X\setminus F)$. In other words, we may assume from the beginning that $m(X\setminus F)=0$.
\par\smallskip
By Fact \ref{factDirac}, there exists a sequence $(\nu_k)_{k\ge 1}$ of elements of $\mathcal P_{T}(X\setminus \overline{F-F}\,)$ such that $\nu_k\to\delta_0$. 
As in the proof of Proposition \ref{continuous2}, set $\mu_k:=\nu_k*m$, the convolution product of $\nu_k$ and $m$. Then $\mu_k\to m$ in $\mathcal P(X)$ because 
$\nu_k\to\delta_0$; and the measures $\mu_k$ are $T$-$\,$invariant because $m$ and $\nu_k$ are $T$-$\,$invariant and $T$ is linear. Finally, let us check that $\mu_k(F)=0$. Since $m(X\setminus F)=0$, we can write 
\begin{eqnarray*}\mu_k(F)&=&\int_X \nu_k\bigl(F -y\bigr)\, dm(y)
=\int_{F} \nu_k \bigl(F -y\bigr)\, dm(y)\, .
\end{eqnarray*}
Now recall that $\nu _{k}$ is supported on $X\setminus \overline{F-F}$, so that in particular $\nu_k(F-y)=0$ for all $y\in F$. So we do have $\mu_k(F)=0$, which concludes the proof of Lemma \ref{residual}.
\end{proof}

\begin{remark}\label{rem4} The linear structure has been used very heavily in the above proof. The linearity of $T$ is needed to show that the measures $\mu_k$ are 
$T$-$\,$invariant, dilations are necessary to approximate the Dirac mass $\delta_0$, and the difference set $F-F$ is essential in the convolution argument. 
As will become clear in the sequel, it 
would be very interesting to find weaker conditions on the set $F$ under which the above reasoning could be carried out, so as to yield that $\mathcal P_{T,*}(X\setminus F)$ is a dense $G_{\delta }$ subset of $\mathcal P_T(X)$. One may observe that what we really need to know on the set $F$ is that  the Dirac mass $\delta_0$ can be approximated by invariant measures $\nu_k$ satisfying $\nu_k(F-y)=0$ for every $y\in F$. This is of course much weaker than requiring $\nu_k(\,\overline{F-F}\,)=0$, but we have not been able to take this into account in order to weaken the assumptions of Proposition \ref{residual} in a satisfactory way.
\end{remark}

\subsection{Further remarks and questions}
One may wonder whether it is always possible to write the set of hypercyclic vectors for a given operator $T$ as $HC(T)=\bigcap_{k\in\NN} (X\setminus F_k)$, where the $F_k$ are closed subsets of $X$ satisfying the assumptions of Lemma \ref{residual}. If it were so, then it would follow immediately that any frequently hypercyclic operator acting on a reflexive Banach space admits an invariant measure $m$ with full support such that $m(HC(T))=1$, and hence and ergodic measure with full support. However, this is \emph{not} always possible. Indeed, the requirement that $F-F$ is nowhere dense in Lemma \ref{residual} is rather strong, because it implies that the set $F$ has to be \emph{Haar-null} in the sense of Christensen, which means that there exists a Borel probability measure $\mu$ on $X$ such that every translate of $F$ has $\mu\,$-$\,$measure $0$. So, if $HC(T)$ can be written as above, then $X\setminus HC(T)$ has to be Haar-null. However, although operators on reflexive Banach spaces for which $X\setminus HC(T)$ is Haar-null do exist (see \cite{GR}), the ``expected" behaviour is rather the opposite one (see \cite{BMM}).

Nevertheless, a positive answer to the next question would imply that any frequently hypercyclic operator acting on a reflexive Banach space admits an ergodic measure with full support. 

\begin{question} Let $T$ be a hypercyclic operator on a Banach space $X$. Is it possible to write $HC(T)$ as $HC(T)=\bigcap_{n\in\NN} (X\setminus F_n)$, where the 
$F_n$ are closed, $T$-$\,$invariant subsets of $X$ and, for each fixed $n\in\NN$, the set $F_n$ has the following smallness property: 
there exists a sequence $(\nu_{k,n})_{k\in\NN}\subset \mathcal P_T(X)$ converging to the Dirac mass $\delta_0$ as $k\to\infty$ such that $\nu_{k,n}(F_n-y)=0$ for all $k$ and every $y\in F_n$?
\end{question}

More modestly, one may also ask the following question.
\begin{question}\label{Haar} Let $T$ be an operator on $X$ admitting an invariant measure with full support. Assume that $X\setminus HC(T)$ is Haar-null. Does it follow that $T$ admit an ergodic measure with full support?
\end{question}

It should be pointed out, however, that the scope of this question may be rather limited. Indeed, even though operators $T$ for which $X\setminus HC(T)$ is 
Haar-null do exist, we do not know whether such operators can admit invariant measures with full support. 

\medskip
To end-up this section, we note that the conclusion of Theorem \ref{th1} is easily seen to be satisfied (without assuming that $X$ is reflexive nor that $T$ is frequently hypercyclic) if the unimodular eigenvectors of $T$ span a dense linear subspace of $X$ and $T^N\neq Id$ for all $N\geq 1$. Indeed, in this case $T$ admits an invariant Gaussian measure $m$ with full support (see \mbox{e.g.} \cite{BM1}), and since any proper Borel subspace of $X$ is \emph{Gauss null}, \mbox{i.e.} negligible for every Gaussian measure with full support (see \mbox{e.g.} \cite{BL}), we have $m({\rm Per}(T))=0$.  
So Theorem \ref{th1} is interesting only for frequently hypercyclic operators having few unimodular eigenvectors. This is an additional motivation for the following question, which was already mentioned in \cite{BG2} in a Hilbert space setting.

\begin{question} Does there exist an operator acting on a reflexive Banach space which is frequently hypercyclic but has no unimodular eigenvalues?
\end{question}

\section{Quantifying the frequent hypercyclicity of an operator}\label{sec4}

\subsection{Introductory remarks} Our starting point in this section is Remark \ref{rem1} at the end of Section \ref{sec2}, which we restate in a linear setting. Let $T$ be a frequently hypercyclic operator on a Banach space 
$X$. For each $R>0$, denote by $B_R$ the closed ball with radius $R$ centered at $0$.
Recall that if $x\in X$ and $B\subset X$, we set
$$\mathcal N_T(x,B):=\{ i\in\NN;\; T^ix\in B\}\, .$$
\par\smallskip
Although this may look counter-intuitive, it is quite possible that a hypercyclic vector $x_0$ for $T$ satisfies
\begin{equation*}
\sup_{R>0}\; \underline{\rm dens}\, \mathcal N_T(x_{0},B_{R}) <1\, , \tag{$*$}
\end{equation*}
or even that
\begin{equation*}
\sup_{R>0}\; \overline{\rm dens}\, \mathcal N_T(x_{0},B_{R}) <1\, . \tag{$**$}
\end{equation*}

Observe that $(*)$ is equivalent to the following property: there exists a set $D\subset\N$ with positive upper density such that $\Vert T^ix_0\Vert\to\infty$ as $i\to\infty$ along 
$D$. Indeed, assume first that $(*)$ holds true, and denote by $c$ the involved supremum (so that $c<1$).  Then, for each positive integer $R$, one can
find infinitely many integers $N\in\N$ such that 
$$\frac1N\,\#\bigl\{ i\in [1,N];\; \Vert T^ix_0\Vert \leq R\bigr\} \leq c +2^{-R}\, .$$
For all such $N$ we have
$$\frac1N\,\#\bigl\{ i\in [1,N];\; \Vert T^ix_0\Vert > R\bigr\}\geq 1-c-2^{-R}\, ,$$ from which it follows that one can find an increasing sequence of integers $(N_{R})_{R\ge 1}$ 
such that
$$\frac{1}{N_{R}}\,\#\bigl\{ i\in (N_{R-1},N_{R}];\; \Vert T^ix_0\Vert > R\bigr\}\geq 1-c-2^{-(R-1)}\, \quad \textrm{ for each }R\ge 1.$$
Then, the set $$D=\bigcup_{R\ge 1}\{i\in (N_{R-1},N_{R}];\; \Vert T^ix_0\Vert > R\}$$
has upper density $\overline{\rm dens} (D)\geq 1-c>0$, and $\Vert T^ix_0\Vert\to\infty$ as $i\to\infty$ along $D$. Conversely, the same argument shows that if there exists a set $D\subset \NN$ with upper density $\overline{\rm dens}(D)=d>0$ such that 
$\Vert T^ix_0\Vert\rightarrow \infty$ as $i\to\infty$ along $D$, then $(*)$ holds true and the involved supremum is not greater than $1-d$. 

This phenomenon turns out to happen quite often. Indeed, building on an argument of 
\cite{BBMP}, Bayart and Ruzsa have shown in \cite[Proposition 14]{BR} that whenever $T$ is a frequently hypercyclic operator on a Banach space $X$, there exists a comeager set of vectors $x\in X$ such that $\Vert T^ix\Vert\to \infty $ as $i\to\infty$ along some set $D_x\subset\NN$ with $\overline{\rm dens}(D_x)=1$. It follows that if $T$ is frequently hypercyclic, then the supremum involved in $(*)$ is actually equal to $0$ for a comeager set of vectors $x_0\in HC(T)$. (Note that such a vector $x_0$ cannot be frequently hypercyclic for $T$, so this  implies in particular that \emph{the set $FHC(T)$ is meager in $X$}; a somewhat different proof of this fact will be given below).

On the other hand, it seems much harder to find operators $T$ such that $(*)$ holds true for some \emph{frequently hypercyclic} vector $x_0$. 
The only known examples are the frequently hypercyclic weighted backward 
shifts on $c_0(\ZZ)$ constructed in \cite[Theorem 7]{BR}, where in fact $(**)$ holds true for all hypercyclic vectors $x_{0}\in X$. 

Observe also that an argument similar  to the one presented above shows the following: $(**)$ holds true for a given vector $x_{0}$, with the involved supremum denoted by $c$, if and only if there exists a set $D\subset\N$ of integers with $\underline{\rm dens}(D)\ge 1- c$ such that 
$\Vert T^ix_0\Vert\rightarrow \infty$ as $i\to\infty$ along $D$.
\par\smallskip
Our aim in this section is to investigate these phenomena a little bit further.

\subsection{The parameter $c(T)$} In this sub-section, we introduce a parameter $c(T)\in [0,1]$ associated with any hypercyclic operator $T$ acting on a Banach space $X$. This parameter measures the maximal frequency with which the orbit of a hypercyclic vector $x$ for $T$ can visit a ball centered at $0$. 

\begin{definition} Let $T$ be a hypercyclic operator on $X$. For each $R>0$, we set $c_R(T):=\sup\limits_{x\in HC(T)}\, \overline{\rm dens}\; \mathcal N_T (x, B_R)$, and we define 
$c(T):=\sup_{R>0}\, c_R(T)$. In other words,
$$c(T)=\sup\limits_{R>0}\; \sup_{x\in HC(T)}\; \overline{\rm dens}\, \mathcal N_{T}(x, B_R)\, .$$
\end{definition}
Note that it is quite natural to exclude the vectors $x$ whose orbits have a ``trivial" behaviour in this definition, since otherwise $c(T)$ would just have no interest at all: for example, 
if $x$ has a bounded $T$-$\,$orbit, then $\mathcal N_T(x,B_R)=\NN$ for all sufficiently large $R$. Note also that if $T$ is frequently hypercyclic (or just upper frequently hypercyclic) then $c(T)>0$.

\smallskip
The following lemma shows that the suprema in the definition of $c(T)$ are attained at a comeager set of points, and also that the above quantity $c_R(T)$ is in fact 
independent of $R>0$.

\begin{lemma}\label{c(T)} For any $\alpha>0$, there exists a comeager set of vectors $x\in HC(T)$ such that $\overline{\rm dens}\, \mathcal N_{T}(x, B_{\alpha})=c(T)$. 
\end{lemma}
\begin{proof} 
The proof relies on the next two facts.

\begin{fact}\label{fait1} Let $z\in HC(T)$, and let $R>0$ be such that $\Vert T^m z\Vert\neq R$ for all $m\geq 0$. Then the set 
$G_{z,R}=\{ x\in HC(T);\; \overline{\rm dens}\, \mathcal N_T(x, B_R)\geq \overline{\rm dens}\, \mathcal N_T(z, B_R)\}$ is comeager in $X$.
\end{fact}

\begin{proof} Set $d_{z,R}=\overline{\rm dens}\, \mathcal N_T(z, B_R)$. For $N\in\NN$ and $\varepsilon>0$, define 
$$ U_{N,\varepsilon}:=\left\{ x\in X;\; \exists n\geq N\; :\; T^ix\not\in\partial B_R\;\hbox{for $i=1,\dots n$ and } 
\frac1n\sum_{i=1}^{n} \mathbf 1_{B_R}(T^ix)> d_{z,R}-\varepsilon\right\}. $$
We denote here by $\partial B_R$ the boundary of the ball $B_R$. The set $U_{N,\varepsilon}$ is open in $X$, because if $n\ge N$ is fixed, any vector $x\in X$ such that 
$T^ix\not\in\partial B_R$ for $i=1,\dots ,n$ is a point of continuity of all functions $\mathbf 1_{B_R}\circ T^i$, $1\leq i\leq n$. Moreover, $U_{N,\varepsilon}$ is also dense in $X$ because it contains the orbit of $z$ under the action of $T$ (and $z$ is hypercyclic). Indeed, if $k\ge 1$ is fixed, then  $T^i(T^kz)$ does not belong to $\partial B_R$ for every 
$i\in\NN$ by the choice of $R$, and also $\mathcal N_T(T^kz,B_R)=\mathcal N_T(z,B_R)-k$, so that
$\overline{\rm dens} \,\mathcal N_T(T^kz,B_R)=\overline{\rm dens} \,\mathcal N_T(z,B_R)=d_{z,R}$.
It follows that $T^kz$ belongs to $ U_{N,\varepsilon}$. By the Baire Category theorem, the set 
$$G:=HC(T)\cap \bigcap_{N,q\ge 1} U_{N,2^{-q}}$$ is a dense $G_\delta$ subset of $X$. 
Since $G$ is obviously contained in $ G_{z,R}$, this concludes the proof.
\end{proof}

\begin{fact}\label{fait2} The quantity $c_R(T)=\sup\limits_{z\in HC(T)}\, \overline{\rm dens}\, \mathcal N_{T}(z, B_R)$ does not depend on $R>0$.
\end{fact}

\begin{proof} By the linearity of $T$ we have $\mathbf 1_{B_R} (T^iz)=\mathbf 1_{B_1} (R^{-1} T^iz)$, so that 
$\mathcal N_{T}(z,B_R)=\mathcal N_T(R^{-1}z, B_1)$ for all $z\in HC(T)$. Since $HC(T)$ is dilation-invariant, this yields at once that $c_R(T)=c_1(T)$ for all $R>0$.
\end{proof}

We are now ready to prove Lemma \ref{c(T)}.
\par\smallskip
Let $\alpha >0$.
By Fact \ref{fait2}, there exists a sequence $(z_p)_{p\ge 1}$ of vectors of $ HC(T)$ such that $\overline{\rm dens}\, \mathcal N_{T}(z_p, B_{\alpha/2})\to c(T)$ as $p\to\infty$. Then, we can choose $R\in (\alpha/2, \alpha)$ such that $\Vert T^m z_p\Vert\neq R$ for each $p$ and $m\geq 0$. By Fact \ref{fait1} and the Baire Category theorem, there exists a comeager set $G\subset HC(T)$ such that 
$\overline{\rm dens}\, \mathcal N_T(x, B_R)\geq \overline{\rm dens}\, \mathcal N_T(z_p, B_R)$ for every $x\in G$ and $p\geq 1$. Since 
$\overline{\rm dens}\, \mathcal N_T(x, B_R)\leq \overline{\rm dens}\, \mathcal N_T(x, B_\alpha)\leq c(T)$ and 
$\overline{\rm dens}\, \mathcal N_T(z_p, B_R)\geq \overline{\rm dens}\, \mathcal N_T(z_p, B_{\alpha/2})$, we obtain by letting $p$ tend to infinity that 
$\overline{\rm dens}\, \mathcal N_T(x, B_\alpha)=c(T)$ for all $x\in G$.
\end{proof}

\begin{remark}\label{reformul} By Lemma \ref{c(T)}, the parameter $c(T)$ may be defined equivalently as follows: for any $\alpha >0$, 
$$c(T)= \max \Bigl\{ c\in [0,1];\; \overline{\rm dens} \;\mathcal{N}_{T}(x,B_{\alpha })\geq c\;\hbox{ for comeager many $x\in X$}\Bigr\}\, .$$
\end{remark}

\begin{remark}\label{remarque} Since the family $(B_R)_{R>0}$ is monotonic with respect to $R$, it follows from Lemma \ref{c(T)} and the Baire Category theorem  that  there is in fact a comeager set of vectors $x\in HC(T)$ such that the following holds true: for every  $\alpha >0$, $\overline{\rm dens}\, \mathcal N_T(x, B_\alpha)=c(T)$.
\end{remark}

\smallskip

\subsection{Two simple applications} In this sub-section, we prove two simple results which illustrate the relevance of the parameter $c(T)$ for the study of the dynamics of an operator $T$.

\smallskip
Our first result shows that any frequently hypercyclic operator has ``distributionally null" orbits. This is a kind of counterpart to \cite[Proposition 14]{BR}, where ``distributionally unbounded" orbits are considered.

\begin{proposition}\label{distrib0} Let $T$ be a hypercyclic operator on a Banach space $X$ with $c(T)>0$. There is then a comeager set of vectors $x\in X$ such that 
$\Vert T^ix\Vert\to 0$ as $i\to \infty$ along some set $D_x\subset\NN$ with $\overline{\rm dens} (D_x)\geq c(T)$.
\end{proposition}

\begin{proof} By Remark \ref{remarque}, we know that the set 
$$G=\{ x\in X;\; \forall \alpha>0\;:\; \overline{\rm dens} \,\mathcal{N}_{T}(x,B_{\alpha })\geq c(T)\}$$ is comeager in $X$. 
It is thus enough to show that all points $x\in G$ have the required property.
Let us fix a vector $x\in G$, and let $(\varepsilon _r)_{r\ge 1}$ be a decreasing sequence of positive numbers tending to $0$. By the definition of $G$, one can find an increasing sequence of integers 
$(n_r)_{r\ge 0}$ (with $n_0=1$) such that 
$$\frac1{n_r}\,  \#\left\{ i\in [1,n_r];\; \Vert T^ix\Vert \leq 1/r\right\} \geq c(T)-\varepsilon _r\quad \textrm{ for all } r\ge 1.$$
 Moreover, extracting if necessary a subsequence from the sequence $(n_r)_{r\ge 0}$, we may also assume  that $(n_r)$ increases very fast; for example that 
$n_{r-1}/n_r\leq \varepsilon _r$ for all $r\geq 1$. This will ensure that 
$$\frac1{n_r}\,  \#\left\{ i\in (n_{r-1},n_r];\; \Vert T^ix\Vert \leq 1/r\right\} \geq c(T)-2\varepsilon _r \quad \textrm{ for all } r\ge 1.$$
If we set 
$$ D_x=\bigcup_{r\geq 1} \bigl\{ i\in (n_{r-1},n_r];\; \Vert T^ix\Vert \leq 1/r\bigr\}\, ,$$
it follows that 
\begin{eqnarray*}
\frac1{n_r}\,\#\Bigl([1,n_r]\cap D_x\Bigr)&\geq &\frac1{n_r} \#\Bigl( (n_{r-1},n_r]\cap D_x\Bigr)
\geq c(T)-2\varepsilon _r \quad \textrm{ for all } r\ge 1,
\end{eqnarray*}
so that $\overline{\rm dens}(D_x)\geq c(T)$. Since obviously $\Vert T^ix\Vert\to 0$ as $i\to\infty$ along $D_x$, this concludes the proof.
\end{proof}

\smallskip
One may wonder if the point $0$ plays any special role in the statement of Proposition \ref{distrib0}. More precisely, is it possible to show that that given any vector $a\in X$, there is a comeager set of vectors $x\in X$ such that $T^ix\to a$ as $i\to\infty$ along some set with positive upper density? This is not so, as shown by the following simple remark.

\begin{remark}\label{proposition} Let $X$ be a Hausdorff topological space, and let $T:X\to X$ be a continuous self-map of $X$. Let also $a\in X$. Assume that one can find 
$x\in X$ such that $T^ix\to a$ as $i\to\infty$ along some set $D\subset\NN$ with $\overline{\rm dens}(D)>0$. Then $a$ is a periodic point of $T$.
\end{remark}

\begin{proof} The key point is  observe that, under the assumptions of Remark \ref{proposition}, one can find an integer $q\ge 1$ such that $D\cap (q+D)$ is infinite. Suppose indeed that it is not the case, and consider the sets $R_k=(k+ D)\setminus \bigcup_{1\leq q<k} (q+D)$, $k\ge 2$.
These sets are pairwise disjoint, and our assumption that $D\cap (k+D)$ is finite for every $k\ge 1$ implies that each set $R_{k}$
is a translate of a cofinite subset of $D$. But this contradicts the fact that $\overline{\rm dens}(D)=c>0$. Indeed, choosing an invariant mean 
$\mathfrak m$  on $\ell^\infty(\NN)$ such that $\mathfrak m(D)=c$, we would then have $\mathfrak m (R_k)\geq c$ for all $k\ge 2$, and hence 
$\mathfrak m\left(\bigcup_k R_k\right)=\infty$, which is impossible.
\par\smallskip
So let $q$ be such that $D\cap (q+D)$ is infinite.
Since $T^ix\to a$ as $i\to\infty$ along $D$ and $T^{i}x\to T^q a$ as $i\to\infty$ along $q+D$, we immediately deduce that $T^q a=a$, so that $a$ is a periodic point of $T$.
\end{proof}

\smallskip
Our second result shows that if $T$ is a frequently hypercyclic operator, then the set $FHC(T)$ is rather small even though it is of course dense in the underlying space $X$. As  mentioned at the beginning of this section, this was already obtained independently in \cite{BR}, and it is also proved 
in Moothathu \cite{M}.

\begin{proposition}\label{FHCmeager} If $T$ is any bounded operator on a Banach space $X$, then $FHC(T)$ is meager in $X$.
\end{proposition}

\begin{proof} We may of course assume that $T$ is hypercyclic. By Lemma \ref{c(T)}, there exists a comeager set of vectors $G\subset HC(T)$ such that  
$\overline{\rm dens}\, \mathcal N_T(x, B_1)=c(T)$ for every $x\in G$. It is enough to show that no vector 
$x\in G$ can be frequently hypercyclic for $T$; so let us fix $x\in G$. Let $V$ be any nonempty open set such that $V\cap B_1=\emptyset$ and 
$V\subset B_2$. Then
\begin{eqnarray*}
\overline{\rm dens} \,\mathcal{N}_{T}(x,B_{2})&\geq&\overline{\rm dens} \,\mathcal{N}_{T}(x,B_{1}) +\underline{\rm dens} \,\mathcal{N}_{T}(x,V)
=c(T)+\underline{\rm dens} \,\mathcal{N}_{T}(x,V)\, .
\end{eqnarray*}
On the other hand, $\overline{\rm dens} \,\mathcal{N}_{T}(x,B_{2})\leq c(T)$ by the definition of $c(T)$. This shows that $\underline{\rm dens} \,\mathcal{N}_{T}(x,V)=0$, and hence that $x$ does not belong to $ FHC(T)$.
\end{proof}

\begin{remark} Essentially the same proof shows that $FHC(T)$ is meager in $X$ for any continuous linear operator $T$ acting on an arbitrary Polish topological vector space 
$X$ (which is the result obtained in \cite{M}). To prove this, one has to modify the definition of the parameter $c(T)$. Take a neighbourhood $B$ of $0$ in $X$ such that $L\cap \partial B$ contains at most one point for every half-line $L$ starting at $0$. Set $B_R:=R\cdot B$ for any $R>0$, and define 
$$c_{B}(T):= \sup\limits_{R>0}\; \sup_{x\in HC(T)}\; \overline{\rm dens}\, \mathcal N_{T}(x, B_R)\, .$$
Then Lemma \ref{c(T)} still holds as stated, with $c_B(T)$ in place of $c(T)$: the proof is exactly the same, once it has been observed that for any sequence 
$(z_p)\subset X$ and every nontrivial interval $I\subset (0,\infty)$, one can find $R\in I$ such that $T^mz_p\not\in\partial B_R$ for each $p$ and all $m\geq 0$ (which is clear because for any countable set $D\subset X$, the set of all $R>0$ such that $a/R\in\partial B$ for some $a\in D$ is countable). Now, just copy out the proof of Proposition 
\ref{FHCmeager}.
\end{remark}

\begin{remark} Proposition \ref{FHCmeager} is not true for general Polish dynamical systems. For example, if $\rho$ is any irrational rotation of the circle $\TT$, then $FHC(\rho)=\TT$.
\end{remark}

\subsection{Proofs of Theorems \ref{th3}, \ref{th4} and \ref{th5}} We are now in position to prove Theorems \ref{th3}, \ref{th4} and \ref{th5} quite easily. We start with Theorem 
\ref{th4}.

 \begin{proof}[Proof of Theorem \ref{th4}] Assume that $T$ admits an ergodic measure $m$ with full support. By the pointwise ergodic theorem, there exists
for each $R>0$ a Borel set $\Omega_R \subset X$ with $m(\Omega_R)=1$ such that, for each $x\in\Omega_R$,
$$\frac1n\sum_{i=1}^n \mathbf 1_{B_R}(T^ix)\To m(B_R)\quad \textrm{ as } {n\to\infty}\, .$$

Let us fix $R>0$ such that $m(\partial B_{R})=0$. As in the proof of Lemma \ref{c(T)}, consider for $N\in\NN$ and $\varepsilon >0$ the open set
\begin{eqnarray*}
 U_{N,\varepsilon}&:=&\Bigl\{ x\in X;\; \exists n\geq N \textrm{ : } T^ix\not\in\partial B_R \textrm{ for } i=1,\dots ,n \\
&&\qquad\qquad\qquad\qquad \textrm{ and } \;\;  
\frac1n\sum_{i=1}^{n} \mathbf 1_{B_R}(T^ix)>m(B_R)- \varepsilon\Bigr\}.
\end{eqnarray*}
Since $m\left(\bigcup_{i\ge 1}T^{-i}(\partial B_{R})\right)=0$ and since $\Omega_R\setminus \bigcup_{i\ge 1}T^{-i}(\partial B_{R})$ is contained in $U_{N\varepsilon}$, we have 
$m(U_{N,\varepsilon})=1$; in particular, the open set  $U_{N,\varepsilon}$ is dense in $X$ because the measure $m$ has full support. By the Baire Category theorem and since any point $x\in \bigcap_{N,q\in\NN} U_{N,2^{-q}}$ satisfies $\overline{\rm dens}\, \mathcal N_T(x,B_R)\geq m(B_R)$, it follows that 
$\overline{\rm dens}\, \mathcal N_T(x,B_R)\geq m(B_R)$ for a comeager set of vectors $x\in X$. Hence we get by Remark \ref{reformul} that 
$c(T)\geq m(B_R)$ for all $R>0$ such that $m(\partial B_{R})=0$. The set of radii $R>0$ such that $m(\partial B_{R})>0$ being at most countable, it follows that 
$c(T)\geq \sup_{R>0}m(B_R)$, i.e. that $c(T)=1$.
\end{proof}

Theorem \ref{th3} can now be deduced from Theorem \ref{th4} and Proposition \ref{distrib0}.

\begin{proof}[Proof of Theorem \ref{th3}] In \cite[Theorem 7]{BR}, the authors construct a frequently hypercyclic bilateral shift $T=B_{\bf w}$ on $c_0(\ZZ)$ enjoying the following property: 
there exists a set $A\subset\NN$ with positive lower density such that, for any $x\in c_0(\ZZ)$, it holds that
$$\forall i\in A\;:\; \Vert T^i x\Vert\geq \vert \langle e_0^*,x\rangle\vert\, .$$ Here 
 $e_0^*$ denotes as usual the $0\,$-$\,$th coordinate functional on $c_0(\ZZ)$, which associates to a vector $x=(x_{n})_{n\in\Z}$ of $c_{0}(\Z)$ the coordinate $x_{0}$.
\par\smallskip
It follows easily that $c(T)\leq 1-\underline{\rm dens} (A)<1$. Indeed, suppose that it is not the case. One can then apply Proposition \ref{distrib0} to get a comeager set of vectors $x\in c_0(\ZZ)$ such that 
$\Vert T^ix\Vert\to 0$ as $i\to\infty$ along some set $D_x\subset\NN$ with $\overline{\rm dens}(D_x)>1-\underline{\rm dens} (A)$. For any such vector $x$, the set $D_x\cap A$ 
has positive upper density, and in particular $D_x\cap A$ is infinite. Since $\Vert T^ix\Vert \geq \vert \langle e_0^*,x\rangle\vert$ for all $i\in A$, it follows that 
$\langle e_0^*,x\rangle=0$. But this implies that the set of all such vectors $x$ cannot be dense in $c_0(\ZZ)$, a contradiction.
\par\smallskip
By Theorem \ref{th4}, we conclude that the operator $T=B_{\bf w}$ does not admit any ergodic measure with full support, even though it is frequently hypercyclic.
\end{proof}

Finally, let us prove Theorem \ref{th5}.
\begin{proof}[Proof of Theorem \ref{th5}] Assume that $T$ admits an ergodic measure with full support. Then $T$ is frequently hypercyclic and $c(T)=1$. By \cite[Proposition 14]{BR}, and since $T$ is frequently hypercyclic, there exists a comeager set $G$ of vectors $x\in X$ such that 
$\Vert T^ix\Vert\to\infty$ as $i\to\infty$ along some set $E_x\subset \NN$ with $\overline{\rm dens} (E_x)=1$. On the other hand, by Proposition \ref{distrib0}, and since $c(T)=1$, there exists also a comeager set $G'$ of vectors $x\in X$ such that $\Vert T^ix\Vert\to 0$ as $i\to\infty$ along some set $D_x\subset\NN$ with $\overline{\rm dens} (D_x)=1$. Then any vector $x$ belonging to the comeager set $ G\cap G'$ is distributionally irregular, which concludes the proof.
\end{proof}

\begin{remark}\label{remarquebis} The weighted shift $B_{\bf w}$ constructed in \cite[Theorem 7]{BR} is in fact not distributionally irregular. So one can alternatively use Theorem \ref{th5} to show that $B_{\bf w}$ does not admit any ergodic measure with full support.
\end{remark}

\section{Invariant measures supported on the set of hypercyclic vectors}\label{sec5} 
In this section, we shall prove Theorem \ref{th6} by using another result (Theorem \ref{th7} below) giving several characterizations, for an operator $T$, of the existence of invariant measures supported on $HC(T)$ and belonging to the closure of some given family $\mathcal M$ of invariant measures. 

\subsection{Adequate families of invariant measures} Let $T$ be a continuous linear operator on a Polish topological vector space $X$. We shall say that a nonempty family of $T\,$-$\,$invariant measures 
$\mathcal M\subset\mathcal P_T(X)$ is \emph{adequate} if it satisfies the following assumptions:
\begin{enumerate}
\item[$\bullet$] $\mathcal M$ is \emph{dilation-invariant}, \mbox{i.e.} for any measure $\mu\in\mathcal M$ and every $r>0$, the measure $\mu^r$ defined by $\mu^r(A)=\mu (r\cdot A)$ for every Borel set $A\subset X$ still belongs to $\mathcal M$;
\item[$\bullet$] $\mathcal M$ is \emph{convolution-invariant}, \mbox{i.e} $\mu_1*\mu_2\in\mathcal M$ whenever $\mu_1,\mu_2\in\mathcal M$;
\item[$\bullet$] each measure $\mu\in\mathcal M$ has \emph{compact support}.
\end{enumerate}

\smallskip
The next lemma gives natural examples of adequate families.

\begin{lemma} The following families of $T\,$-$\,$invariant measures are adequate:
\begin{enumerate}
\item[\rm (a)] $\mathcal F_T(X)$, the convex hull of the family of all periodic measures for $T$;
\item[\rm (b)] $\mathcal S_T(X)$, the family of all Steinhaus measures for $T$;
\item[\rm (c)] the family of all compactly supported $T\,$-$\,$invariant measures.
\end{enumerate}
\end{lemma}
\begin{proof} We first note that $\mathcal P_T(X)$ is convolution-invariant; see the proof of Proposition \ref{continuous2}.

\smallskip
(a) Recall that a periodic measure for $T$ is a measure of the form
$$\nu_a =\frac1N\sum_{i=0}^{N-1} \delta_{T^ia}\, ,$$
where $N\geq 1$ and $a\in X$  satisfy $T^Na=a$. So ${{\mathcal F}}_T(X)={\rm conv}\, \{ \nu_a;\; a\in{\rm Per}(T)\}$. 
By Remark \ref{bi-obvious}, ${{\mathcal F}}_{T}(X)$ can be equivalently defined as the set of all finitely supported, $T$-$\,$invariant measures. Hence (and since $\mathcal P_T(X)$ is convolution-invariant) it is clear that $\mathcal F_T(X)$ is adequate.

\smallskip
(b) Recall that a measure $\mu\in\mathcal P(X)$ is a 
{Steinhaus measure for $T$} if $\mu$ is the distribution of an $X$-valued random variable of the form
$\Phi(\omega)=\sum_{j\in J} \chi_j(\omega)\, x_j$,
where the $x_j$ are unimodular eigenvectors for $T$ and $(\chi_j)_{j\in J}$ is a finite sequence of independent {Steinhaus variables}. In this case we write 
$\mu\sim \sum_{j\in J} \chi_j x_j$. 

It is clear that Steinhaus measures have compact support and that $\mathcal S_T(X)$ is dilation-invariant. So we just have to check that 
$\mathcal S_T(X)$ is convolution-invariant. But this is also clear. Indeed, if $\mu_1, \mu_2\in\mathcal S_T(X)$, we may write $\mu_1\sim\sum_{j\in J_1}\chi_j x_j$ and 
$\mu_2\sim\sum_{j\in J_2}\chi_j x_j$, where the index sets $J_1$, $J_2$ are disjoint and the Steinhaus variables $\chi_j$, $j\in J_1\cup J_2$, are independent. Then the random variables $\Phi_1=\sum_{j\in J_1}\chi_j x_j$ and $\Phi_2=\sum_{j\in J_2}\chi_j x_j$ are independent; so the measure $\mu=\mu_1*\mu_2$ is the distribution of the random variable 
$\Phi=\Phi_1+\Phi_2=\sum_{j\in J_1\cup J_2} \chi_j x_j\, ,$
and hence $\mu_1*\mu_2$ is a Steinhaus measure for $T$.

\smallskip
(c) This is clear since $\mathcal P_T(X)$ is convolution-invariant and the convolution of two compactly supported measures is again compactly supported.
\end{proof}

\subsection{Invariant measures in the closure of an adequate family} For any family of measures $\mathcal M\subset\mathcal P(X)$, we shall denote by $\overline{\mathcal M}$ the closure of $\mathcal M$ in $\mathcal P(X)$.  Recall also that for any Borel set $A\subset X$, we denote by $\mathcal P(A)$ the family of all probability measures $m$ on $X$ such that $m(A)=1$. Accordingly, we set
$${\overline{\mathcal M}}(A):={\overline{\mathcal M}}\cap \mathcal P(A)=\{ m\in{\overline{\mathcal M}};\; m(A)=1\}\, .$$

Note that this is \emph{not} the closure of $\mathcal M(A):=\{m\in\mathcal M;\; m(A)=1\}$ in $\mathcal P(X)$.

\smallskip
The following result, which will be our main tool for proving Theorem \ref{th6}, gives several equivalent conditions for the existence of an invariant measure supported on $HC(T)$ and belonging to the closure of some given adequate family of $T\,$-$\,$invariant measures.  Here and afterwards, we denote by $\mathcal O_T$ the family of all backward $T\,$-$\,$invariant, nonempty open subsets of $X$.

\begin{theorem}\label{th7} Let $T$ be a continuous linear operator on a Polish topological vector space $X$, and let $\mathcal M\subset\mathcal P_T(X)$ be an adequate family of $T\,$-$\,$invariant measures. Consider the following assertions.
\begin{enumerate}
\item[\rm (1)] $T$ admits an ergodic measure $m$ with full support such that $m\in {\overline{\mathcal M}}$;
\item[\rm (2)] $T$ admits an invariant measure $m$ such that $m\in{\overline{\mathcal M}}$ and $m(HC(T))=1$;
\item[\rm (3)] for any $\Omega\in\mathcal O_T$,  the Dirac measure $\delta_0$ belongs to the closure of
$\overline{\mathcal M}(\Omega)$ in $\mathcal P(X)$;
\item[\rm (4)] for any $O\in\mathcal O_T$,  the set $\overline{\mathcal M}(O)$ is a dense $G_{\delta }$ subset of $\overline{\mathcal M}$;
\item[\rm (5)] the set $\overline{\mathcal M}(HC(T))$ is a dense $G_{\delta }$ subset of $\overline{\mathcal M}$.
\end{enumerate}
Then, assertions $(2)$ to $(5)$ are equivalent and implied by {\rm (1)}. Moreover, if assertions $(2)$ to $(5)$ hold, then $T$ admits an ergodic measure with full support.
\end{theorem}

\begin{proof}
The most interesting part of the proof is to show that (3) implies (4). We start with the straightforward implications.
\par\smallskip
$(1)\implies (2)$. It is enough to show that if $m$ is any ergodic measure for $T$ with full support, then $m(HC(T))=1$. This is standard, but we repeat the argument anyway. Let 
$(V_j)_{j\geq 1}$ be a countable basis of (nonempty) open sets for $X$, and set $O_j=\bigcup_{n\geq 0} T^{-n} (V_j)$. Then $HC(T)=\bigcap_{j\geq 1} O_j$. Moreover, we have $T^{-1}(O_j)\subset O_j$, so that $m(T^{-1}(O_j)\Delta O_j)=0$ by the $T$-$\,$invariance of $m$. By ergodicity and since $m(O_j)\geq m(V_j)>0$, it follows that $m(O_j)=1$ for all 
$j\ge 1$, and hence $m(HC(T))=1$. 
\par\smallskip
$(2)\implies (3)$. Assume that (2) holds true for some measure $m\in \overline{\mathcal M}$. For each $n\ge 1$, consider the measure $\mu_n$ defined by setting 
$\mu_n(A)=m(2^n \cdot A)$ for every Borel set $A\subset X$. Then $\mu_n\in\overline{\mathcal M}$ because $\mathcal M$ is dilation-invariant and dilations are continuous on $\mathcal P(X)$. Moreover, since $HC(T)$ is dilation-invariant we have $\mu_n(HC(T))=m(HC(T))=1$. Hence $\mu_n(\Omega )=1$ for any set $\Omega \in\mathcal O_{T}$ as well, since any (nonempty) backward $T$-$\,$invariant open set necessarily contains $HC(T)$. Finally, since 
$\int_X f\, d\mu_n=\int_X f(2^{-n} x)\, dm(x)$ for every $f\in\mathcal C_b(X)$, it is clear that $\mu_n\to\delta_0$ as $n\to\infty$.
\par\smallskip
$(4)\implies (5)$. We use the same notation as in the proof of the implication $(1)\implies (2)$. For each $j\ge 1$, the open set 
$O_j$ belongs to $ \mathcal O_T$, so that $\overline{\mathcal M} (O_j)$ is a dense $G_{\delta }$ subset of $\overline{\mathcal M}$. Hence, 
$\overline{\mathcal M}(HC(T))=\bigcap_{j\geq 1} \overline{\mathcal M}(O_j)$ is a dense $G_{\delta }$ subset of $\overline{\mathcal M}$ as well.
\par\smallskip
$(5)\implies (2)$ is obvious.

\smallskip
Finally, let us recall why (5) implies the existence of an ergodic measure with full support. If (5) holds then $\overline{\mathcal M}(HC(T))$ is in particular non empty. Applying the ergodic decomposition theorem to any measure
$\mu\in\overline{\mathcal M}(HC(T))$, we see that there exists at least one ergodic measure $m$ for $T$ such that $m(HC(T))>0$. By $T\,$-$\,$invariance and since 
$HC(T)\subset \bigcup_{n\geq 0} T^{-n}(U)$ for every open set $U\neq\emptyset$, the measure $m$ has full support.
\par\smallskip
We now turn to the implication $(3)\implies (4)$. As in the proof of Theorem \ref{th2}, this will rely on a specifically linear convolution argument. 
\par\smallskip
Let us fix $O\in\mathcal O_T$. By Remark \ref{P(O)}, we know that $\mathcal P(O)$ is a $G_\delta$ subset of $\mathcal P(X)$, so that $\overline{\mathcal M}(O)$ is a 
$G_\delta$ subset of $\overline{\mathcal M}$. What has to be shown is that $\overline{\mathcal M}(O)$ is dense in $\overline{\mathcal M}$; and to do this, it is enough to prove that  any measure $\nu\in \mathcal M$ belongs to the closure of $\overline{\mathcal M}(O)$ in $\mathcal P(X)$. Let us fix such a measure $\nu$, and set $K:={\rm supp}(\nu)$, so that $K$ is a compact subset of $X$ by the definition of an adequate family.

\begin{fact}\label{OT} The set $\Omega:=\bigcap_{y\in K} (O-y)$ belongs to $\mathcal O_T$.
\end{fact}
\begin{proof} We first check that $\Omega$ is open in $X$. Set $F:=X\setminus O$ and $$C:=\{ (x,y)\in X\times K;\; x+y\in F\}\, .$$ Then $C$ is a closed subset of $X\times K$, so the set 
$\pi_X(C)=\{ x\in X;\; \exists y\in K\;:\; (x,y)\in C\}$ is closed in $X$, being the projection of a closed subset of $X\times K$ along the compact factor $K$. Since 
$\pi_X(C)=X\setminus\Omega$, this shows that 
$\Omega$ is indeed open in $X$.

\smallskip
Next, we observe that $T(K)=K$. Indeed, if $V$ is an open set such that $V\cap T(K)\neq\emptyset$, then $T^{-1}(V)$ is an open set such that $T^{-1}(V)\cap K\neq\emptyset$ and hence $\nu(T^{-1}(V))>0$. Since $\nu$ is $T\,$-$\,$invariant,  it follows that $\nu(V)>0$ for every open set $V$ such that $V\cap T(K)\neq\emptyset$; in other words, 
$T(K)\subset K$. Conversely, since $T(K)$ is compact the set $V:=X\setminus T(K)$ is open in $X$, and we have $\nu(V)=\nu(T^{-1}(V))=0$ because 
$T^{-1}(V)\cap K=\emptyset$. Hence, $X\setminus T(K)\subset X\setminus K$, \mbox{i.e} $K\subset T(K)$.

\smallskip Using this, it is easy to show that $T^{-1}(\Omega)\subset \Omega$. Indeed, since $T(K)=K$ we can write the open set $\Omega$ as $\Omega=\bigcap_{y\in K} (O-T(y))$. Hence, 
$$
T^{-1}(\Omega)=\bigcap_{y\in K} T^{-1}(O-T(y))=\bigcap_{y\in K} (T^{-1}(O)-y)\subset \bigcap_{y\in K} (O-y)=\Omega\, .
$$
(Note that we have used the linearity of $T$ for the second equality). This proves Fact \ref{OT}.
\end{proof}

\begin{fact}\label{convol} If $m\in\mathcal P(\Omega)$, then $m*\nu\in\mathcal P(O)$.
\end{fact}
\begin{proof} Set $F:=X\setminus O$. Then $m(F-y)=0$ for every $y\in K$ because $\Omega\cap (F-y)=\emptyset$, and hence
$$(m*\nu)(F)=\int_K m(F-y)\, d\nu(y)=0\, , $$ which proves Fact \ref{convol}.
\end{proof}
It is now easy to conclude the proof. By (3) and Fact \ref{OT}, one can find a sequence $(m_n)\subset \overline{\mathcal M}(\Omega)$ such that $m_n\to\delta_0$. By the convolution-invariance of $\mathcal M$, the separate continuity of the convolution product  and Fact \ref{convol}, the measures $\nu_n:= m_n*\nu$ belong to 
$\overline{\mathcal M}(O)$. Hence, $\nu=\lim\nu_n$ belongs to the closure of $\overline{\mathcal M}(O)$ in $\mathcal P(X)$.
\end{proof}

The proof of Theorem \ref{th7} is now complete.
\qed

\smallskip
Applying Theorem \ref{th7} to the family $\mathcal M$ of all $T\,$-$\,$invariant measures with compact support, we get the following result.
\begin{corollary}\label{shouldbeuseful0} Let $T$ be a continuous linear operator on $X$. Assume that for any backward $T\,$-$\,$invariant open set $\Omega\neq\emptyset$ and any neighbourhood $W$ of $0$, one can find a $T\,$-$\,$invariant probability measure $\mu$ with compact support such that $\mu(\Omega)=1$ and $\mu(W)$ is arbitrarily close to $1$. Then $T$ admits an ergodic measure with full support.
\end{corollary}

\smallskip
Thus, the operator $T$ admits an ergodic measure with full support as soon as it admits compactly supported invariant measures concentrated around $0$ and supported on any backward  $T\,$-$\,$invariant open set. This condition does not appear to be that strong. In particular, it may possibly be satisfied by all \emph{chaotic} operators, since a chaotic operator admits a great supply of finitely supported invariant measures (namely, the periodic measures). So we cannot exclude the possibility that every chaotic operator admits an ergodic measure with full support (and hence is frequently hypercyclic). Note that if $T$ is chaotic and if $(a_r)_{r\geq 1}$ is a dense sequence of periodic vectors for $T$, then 
$\nu=\sum_{r\geq 1} 2^{-r} \nu_{a_r}$
is a $T$-$\,$invariant measure with full support. However, this measure is obviously not ergodic!

\subsection{Ergodic measures with a finite moment} Recall that a Banach space $X$ is said to have \emph{type $p\in [1,2]$} if there exists some constant $C<\infty$ such that
$$ \mathbb E\left\Vert \sum_n\varepsilon_n x_n\right\Vert^p\leq C \sum_n \Vert x_n\Vert^p$$
for any finite sequence $(x_n)$ in $X$, where the $\varepsilon_n$ are independent Rademacher variables. It is a well known fact that if $X$ has {type $2$}, then the search for invariant \emph{Gaussian} measures for an operator $T\in\mathfrak L(X)$ is in some sense equivalent to the search for invariant measures $m$ admitting a second order moment. More precisely, to any 
$T\,$-$\,$invariant measure $m$ such that $\int_X \Vert x\Vert^2 dm(x)<\infty$ one can associate in a canonical way a Gaussian invariant measure $\widetilde m$, which has the same support and the same ergodicity properties as $m$; see \mbox{e.g.} \cite[Remark 8.2]{BM2}.  

\smallskip
Therefore (and even without assuming that $X$ has type $2$) it is quite natural to look for conditions ensuring that an operator 
$T\in\mathfrak L(X)$ admits an ergodic measure $m$ such that $\int_X \Vert x\Vert^2 dm(x)<\infty$ or, more generally, such that $\int_X\Vert x\Vert^p dm(x)$ for some given $p\in (0,\infty)$. This is the content of the next theorem.

\begin{theorem}\label{th8} Let $T$ be a bounded operator on a Banach space $X$, and let $\mathcal M\subset\mathcal P_T(X)$ be an adequate family of $T\,$-$\,$invariant measures.  Let also $p\in (0,\infty)$. Assume that there exists some finite constant $C$ such that the following holds: 

for every backward $T\,$-$\,$invariant open set $\Omega\neq\emptyset$, one can find a measure $\mu\in{\mathcal M}(\Omega)$ such that $\int_X\Vert x\Vert^p d\mu(x)\leq C$. 

Then, there exists a measure $m\in \overline{\mathcal M}(HC(T))$ such that 
$\int_X \Vert x\Vert^p dm(x)<\infty$. In particular, the operator $T$ admits an ergodic measure $\mu$ with full support such that $\int_X \Vert x\Vert^p d\mu(x)<\infty$.
\end{theorem}
\begin{proof} The proof is very similar to that of Theorem \ref{th7}, so we just briefly indicate the modifications that are needed.

\smallskip
We first note that a simple dilation argument allows us to strengthen the assumption made on $T$.

\begin{fact}\label{autorenforcement}
For every backward $T\,$-$\,$invariant open set $\Omega\neq\emptyset$ and any $\varepsilon>0$, one can find a measure $\mu\in{\mathcal M}(\Omega)$ such that 
$\int_X\Vert x\Vert^p d\mu(x)< \varepsilon$
\end{fact}
\begin{proof} Let us choose $\eta>0$ such that $C\eta^p<\varepsilon$. Then, consider the backward $T\,$-$\,$invariant open set $\Omega_\eta:= \eta\cdot \Omega$. By assumption, one can find a measure $\mu_\eta\in\overline{\mathcal M}(\Omega_\eta)$ such that $\int_X\Vert x\Vert^p d\mu_\eta(x)\leq C$. Now, let $\mu$ be the dilated measure defined by 
$\mu(A):=\mu_\eta (\frac1\eta\cdot A)$ for every Borel set $A\subset X$. Then $\mu\in{\mathcal M}$ because ${\mathcal M}$ is dilation-invariant, and $\mu$ is supported on 
$\Omega$. Finally, we have 
$$\int_X\Vert x\Vert^p d\mu(x)=\int_X \left\Vert \eta x\right\Vert^p d\mu_\eta(x)\leq C\eta^p<\varepsilon\, . $$
\end{proof}

\smallskip
Now, let us denote by ${\mathcal M}^1$ the set of all measures $\mu\in{\mathcal M}$ such that $\int_X \Vert x\Vert^p d\mu(x)< 1$. This set is nonempty by the previous Fact.   We are going to show that $\overline{\mathcal M^1}(HC(T))$ is a dense $G_\delta$ subset of $\overline{\mathcal M^1}$, where $\overline{\mathcal M^1}$ is the closure of 
${\mathcal M}^1$ in $\mathcal P(X)$.

\smallskip
To this end, it is enough to show that for any backward $T\,$-$\,$invariant open set $O\neq \emptyset$, the set ${\mathcal M}^1(O)$ is dense in $\overline{\mathcal M^1}$.
Following the proof of Theorem \ref{th7}, this will be done in two steps: we first show that for any backward $T\,$-$\,$invariant open set $\Omega\neq \emptyset$, the Dirac mass 
$\delta_0$ belongs to the closure of ${\mathcal M^1}(\Omega)$; and then we prove the required result by using a convolution argument.

\smallskip
The first step relies on the following simple consequence of Markov's inequality.
\begin{fact}\label{Markov}
Let $W$ be a neighbourhood of $0$ in $X$, and let $\varepsilon >0$. Then one can find $\eta >0$ such that the following holds: if $\mu\in\mathcal P(X)$ satisfies 
$\int_X \Vert x\Vert^pd\mu(x)<\eta$, then $\mu(W)>1-\varepsilon$.
\end{fact}
\begin{proof} Choose $\alpha >0$ such that $B(0,\alpha)\subset W$. If $\mu\in\mathcal P(X)$ then, by Markov's inequality, we have $\mu(X\setminus B(0,\alpha))\leq \frac1{\alpha^p}\times \int_X \Vert x\Vert^pd\mu(x)$. 
So one can take $\eta:=\alpha^p\varepsilon$.
\end{proof}

\smallskip
From Facts \ref{autorenforcement} and \ref{Markov}, we immediately get 
\begin{fact}\label{fact1deplus} Given a backward $T\,$-$\,$invariant open set $\Omega\neq \emptyset$ and $\varepsilon >0$, one can find a sequence $(m_n)\subset {\mathcal M^1}(\Omega)$ 
such that $m_n\to\delta_0$ and $\int_X \Vert x\Vert^p dm_n(x)<\varepsilon$ for all $n\in\NN$.
\end{fact}

For the second step, an examination of the proof of the implication $(3)\implies (4)$ in Theorem \ref{th7} reveals that it is enough to prove the following Fact.
\begin{fact}\label{fact2deplus} Let $\nu\in \mathcal M^1$. Then one can find $\varepsilon >0$ such that $m*\nu\in{\mathcal M^1}$ for any measure $m\in {\mathcal M^1}$ satisfying 
$\int_X \Vert x\Vert^p dm(x)<\varepsilon$. 
\end{fact}
\begin{proof} Since $\mathcal M$ is convolution-invariant, it is enough to find $\varepsilon>0$ such that, whenever $m\in\mathcal P(X)$ satisfies $\int_X \Vert x\Vert^p dm(x)<\varepsilon$, it follows that $\int_X\Vert x\Vert^p d(m*\nu)(x)<1$. 

Set $\alpha:=\int_X \Vert y\Vert^pd\nu(y)$, so that $\alpha <1$. If $p\geq 1$ then, for any $m\in\mathcal P(X)$, we have
\begin{eqnarray*}
\int_X\Vert x\Vert^p d(m*\nu)(x)&=&\int_{X\times X} \Vert x+y\Vert^p dm(x)d\nu(y)\\
&\leq&\int_{X\times X} (\Vert x\Vert+ \Vert y\Vert)^p\, dm(x)d\nu(y)\\
&\leq& \left[\left(\int_{X\times X} \Vert x\Vert^p dm(x)d\nu(y)\right)^{1/p}+\left( \int_{X\times X} \Vert y\Vert^p dm(x)d\nu(y)\right)^{1/p}\right]^p\\
&<&\left(\varepsilon^{1/p}+\alpha^{1/p}\right)^p\, .
\end{eqnarray*}
Hence, it is enough to take $\varepsilon$ such that $\varepsilon^{1/p}+\alpha^{1/p}\leq 1$. 

If $p<1$ this is even simpler: since $(\Vert x\Vert+\Vert y\Vert)^p\leq \Vert x\Vert^p+\Vert y\Vert^p$ we get 
$$ \int_X\Vert x\Vert^p d(m*\nu)(x)\leq \int_{X\times X} \Vert x\Vert^p\, dm(x)d\nu(y)+\int_{X\times X} \Vert y\Vert^p\, dm(x)d\nu(y)<\varepsilon +\alpha\, ,$$
so we may take $\varepsilon:=1-\alpha$.
\end{proof}

\smallskip
By Fact \ref{fact1deplus}, Fact \ref{fact2deplus} and the proof of Theorem \ref{th7}, we can now conclude that $\overline{\mathcal M^1}(HC(T))$ is a dense $G_\delta$ subset of 
$\overline{\mathcal M^1}$. In particular, $\overline{\mathcal M^1}(HC(T))$ is nonempty. Moreover, since the set 
$\{ \mu\in\mathcal P(X);\; \int_X \Vert x\Vert^p d\mu(x)\leq 1\}$ is closed in $\mathcal P(X)$, any measure $m\in \overline{\mathcal M^1}$ satisfies 
$\int_X \Vert x\Vert^p dm(x)\leq 1$. Finally, applying the ergodic decomposition theorem to any $m\in \overline{\mathcal M^1}(HC(T))$, we get an ergodic measure $\mu$ with full support such that $\int_X \Vert x\Vert^p d\mu(x)\leq 1$.
\end{proof}

\smallskip
If we take $\mathcal M$ to be the family of all compactly supported $T\,$-$\,$invariant measures, we get the following variant of Corollary \ref{shouldbeuseful0}.
\begin{corollary}\label{shouldbeuseful0'} Let $p\in (0,\infty)$. Assume that there exists some finite constant $C$ such that the following holds:
 for every backward $T\,$-$\,$invariant open set $\Omega\neq\emptyset$, one can find a $T\,$-$\,$invariant, compactly supported probability measure $\mu$ such that $\mu(\Omega)=1$ and $\int_X\Vert x\Vert^p d\mu(x)\leq C$. 
 Then $T$ admits an ergodic measure with full support $m$ such that 
$\int_X \Vert x\Vert^p dm(x)<\infty$.
\end{corollary}

\subsection{Steinhaus measures and proof of Theorem \ref{th6} (b)} As it turns out, it is shown in \cite{G} that if an operator $T\in\mathfrak L(X)$ has a perfectly spanning set of unimodular eigenvectors, then the equivalent conditions $(2)$ to $ (5)$ of  Theorem \ref{th7} \emph{are satisfied} for the family $\mathcal S_T(X)$ of all Steinhaus measures for $T$:

\begin{theorem}[\cite{G}]\label{blackbox} Let $T$ be an operator on a complex Banach space $X$, and assume that $T$ has a perfectly spanning set of unimodular eigenvectors. Then $T$ admits an invariant measure with full support $m$ which satisfies $m(HC(T))=1$ and belongs to $\overline{\mathcal S}_T(X)$.
\end{theorem}

\smallskip
Actually, the result is not stated exactly in this way in \cite{G}. What is precisely proved in \cite{G} is that $T$ admits an invariant measure $m$ with full support such that $m(HC(T))=1$ and, moreover, $m$ is the distribution of an $X$-$\,$valued almost surely convergent random series 
$$\Phi(\omega)=\sum_{n=1}^\infty \Phi_n(\omega)$$
defined on some standard probability space $(\Omega,\mathfrak F,\PP)$, where the random variables $\Phi_n$ have the following form: 
$$\Phi_n(\omega)=\sum_{j\in J_n} \chi_j(\omega )\, x_{j}$$
where the sets $J_n$ are pairwise disjoint finite subsets of $\NN$, the vectors $x_j$, $j\in J_n$, are unimodular eigenvectors of $T$, and $(\chi_j)_{j\ge 1}$ is a sequence of independent Steinhaus variables defined on $(\Omega ,\mathfrak{F}, \P)$. To get Theorem \ref{blackbox} as stated, it remains to show that $m\in\overline{\mathcal S}_T(X)$; but this is clear. Indeed, for each $N\geq 1$, set $S_N(\omega):=\sum_{n=1}^N \Phi_n(\omega)$ and denote by $m_N$ the distribution of the random variable $S_N$. Then $m_N$ is a Steinhaus measure for $T$,  and since almost sure convergence implies convergence in distribution, the measures $m_N$ converge to $m$ in $\mathcal P(X)$ as $N\to\infty$. Hence $m$ belongs to $\overline{\mathcal S}_T(X)$.

\begin{remark}\label{secondmoment} It is shown in \cite{G} that a measure $m$ of the above form, \mbox{i.e.} the distribution of a random Steinhaus series $\sum \chi_j x_j$ where the $x_j$ are unimodular eigenvectors for $T$,  can never be ergodic for $T$. However, by the ergodic decomposition theorem one can deduce directly from the existence of $m$ that $T$ admits an ergodic measure $\mu$ with full support. Of course, this measure $\mu$ has no reason at all for being Gaussian (whereas we \emph{know}  by \cite{BM2} that under these assumptions ergodic Gaussian measures with full support do exist), and $T$ may not be weakly mixing with respect to it. Still, an interesting observation is that
 the \mea\ $m$ satisfies $\int_X \Vert x\Vert^2 dm(x)<\infty$.
 Indeed, using the notation above, the construction of \cite{G} is carried out in such a way that $\mathbb{E}\Vert\Phi_n\Vert^{2}<4^{-n}$ for each $n\ge 1$. So $$\mathbb{E}\Vert\Phi\Vert^{2}\le\sum_{n\ge 1} 2^{n}\,\mathbb{E}\Vert\Phi_n\Vert^{2}<\sum_{n\ge 1}2^{-n}<\infty.$$ It follows from this observation that
 the \erg\ \mea\ $\mu$   can be chosen in such a way that  $\int_X \Vert x\Vert^2 d\mu(x)<\infty$ as well; so we may in fact be not very far from Gaussian measures.
\end{remark}

\smallskip
By Theorem \ref{th7}, part (b) of Theorem \ref{th6} follows immediately from Theorem \ref{blackbox}. However, the construction of the above measure $m$ performed in \cite{G} is rather complicated. Moreover, it is somewhat unsatisfactory that only the ``trivial" implications in Theorem \ref{th7} have been used in this argument. 

We now show that in fact, one can prove Theorem \ref{th6} (b) as stated in a much simpler way. By the ``nontrivial" part of Theorem \ref{th7}, this gives in turn a new (and in our opinion, simpler) proof of Theorem \ref{blackbox}. Altogether, we therefore obtain a ``soft" Baire category proof of the fact that any operator with a perfectly spanning set of unimodular eigenvectors admits an ergodic measure 
with full support.

\begin{proof}[Baire category proof of Theorem \ref{th6} (b)] Let $T\in\mathfrak L(X)$ have a perfectly spanning set of unimodular eigenvectors. By Theorem \ref{th7}, condition (3) applied with $\mathcal M=\mathcal S_T(X)$, it is enough to show that 
for any backward $T\,$-$\,$invariant open set $\Omega\neq\emptyset$ and any neighbourhood $W$ of $0$, there exists a Steinhaus measure $\mu$ for $T$ supported on 
$\Omega$ such that $\mu(W)$ is arbitrarily close to $1$. This will rely on the next three facts.

\begin{fact}\label{factperf} Let us denote by $\mathcal E_{\rm cond}(T)$ the set of all unimodular eigenvectors $v$ for $T$ satisfying the following property: for every neighbourhood $V$ of $v$, there are uncountably many 
$\lambda\in\TT$ such that $V\cap\ker(T-\lambda I)\neq\emptyset$. Then the linear span of $\mathcal E_{\rm cond} (T)$ is dense in $X$.
\end{fact}

This follows from, and in fact is equivalent to, the assumption that $T$ has a perfectly spanning set of unimodular eigenvectors. See \mbox{e.g.} \cite[Lemma 3.6]{BM2}, or the proof of \cite[Proposition 4.1]{G}; see also Proposition 
\ref{psp} below. The strange notation $\mathcal E_{\rm cond}$ is explained in Section \ref{finalsection}.

\begin{fact}\label{factind} Let $(v_1,\dots ,v_N)$ be a finite family of vectors in $\mathcal E_{\rm cond} (T)$. Then one can find a family of unimodular eigenvectors 
$(u_1,\dots ,u_N)$ as close as we wish to $(v_1,\dots ,v_N)$ such that the associated family of eigenvalues $(\lambda_1,\dots ,\lambda_N)$ is \emph{independent}; that is, 
the only solution of the equation $\lambda_1^{m_1}\cdots \lambda_{N}^{m_N}=1$ with $m_i\in\ZZ$ is $m_1=\dots =m_N=0$.
\end{fact}
\begin{proof} Since $v_1\in\mathcal E_{\rm cond} (T)$ and the set of all roots of unity is countable, one can find a unimodular eigenvector $u_1$ arbitrarily close to $v_1$ whose associated eigenvalue 
$\lambda_1$ is not a 
root of unity, \mbox{i.e.} the $1\,$-$\,$element family $(\lambda_1)$ is independent. Likewise, since $v_2\in\mathcal E_{\rm cond} (T)$ and the set $\{ \lambda\in\TT;\; (\lambda_1,\lambda)\;\hbox{is not independent}\}$ is countable, one can find $u_2\in\mathcal E(T)$ with associated eigenvalue $\lambda_2$ such that $u_2$ is very close to $v_2$ and the family $(\lambda_1,\lambda_2)$ is independent. Continuing in this way, the result follows.
\end{proof}

\begin{fact}\label{factomega} Let $\Omega\subset X$ be a backward $T\,$-$\,$invariant open set, and let $(u_k)_{k\in K}$ be a finite family of unimodular eigenvectors for $T$. Assume that the associated family of eigenvalues $(\lambda_k)_{k\in K}$ is independent, and that $\sum_{k\in K} u_k\in \Omega$. Then $\sum_{k\in K} \mu_k u_k\in \Omega$ for every $(\mu_k)_{k\in K}\in\TT^K$.
\end{fact}
\begin{proof} Let us fix $(\mu_k)_{k\in K}\in\TT^K$. Since the family $(\lambda_k)_{k\in K}$ is independent, one can apply {Kronecker's theorem} to find $n\in\NN$ such that $\lambda_k^{-n}$ is so close to $\mu_k$ for all $k\in K$  that $\sum_{k\in K} \lambda_k^n\mu_k u_k\in\Omega$. Since $\sum_{k\in K} \lambda_k^n\mu_k u_k=T^n\left( \sum_{k\in K} \mu_k u_k\right)$ and $\Omega$ is backward $T\,$-$\,$invariant, it follows that $\sum_{k\in K} \mu_k u_k\in\Omega$.
\end{proof}

\smallskip
Now let us fix a backward $T\,$-$\,$invariant open set $\Omega\neq\emptyset$ and a neighbourhood $W$ of $0$. Recall that we are looking for a Steinhaus measure $\mu$ for $T$ supported on $\Omega$ such that $\mu(W)$ is close to $1$. By Markov's inequality (see Fact \ref{Markov} in the proof of Theorem \ref{th8}) it is enough to find $\mu\in\mathcal S_T(\Omega)$ such that $\int_X \Vert x\Vert^2d\mu(x)$ is 
 small; say $\int_X \Vert x\Vert^2d\mu(x)<\eta$ for some given $\eta >0$.

\smallskip
 The proof makes use of a simple, yet crucial idea from \cite{G}, namely that of replacing one unimodular eigenvector $x$ by a finite sum $\sum_j a_j u_j$, where the $u_j$ are unimodular eigenvectors associated with independent eigenvalues and $\sum_j\vert a_j\vert^2$ is very small.

\smallskip
By Fact \ref{factperf} and since $\mathcal E_{\rm cond} (T)$ is obviously invariant under scalar multiplication, one can find $x_1,\dots ,x_I\in \mathcal E_{\rm cond} (T)$ such that $x:=\sum_{i} x_i$ lies in $\Omega$. 

\smallskip
Since the linear span of the vectors $x_1,\dots x_I$ is finite-dimensional, one can find some finite constant $M=M(x_1,\dots x_I)$ such that $\left\Vert\sum_{i} \theta_i x_i\right\Vert^2\leq M\, \sum_{i} \vert\theta_i\vert^2$ for every sequence of scalars $(\theta_i)_{i=1}^I$. Having fixed $C$ in this way, we choose $a_1,\dots ,a_J\in(0,\infty)$ such that 
$\sum_j a_j=1$ and $ MI\, \sum_j\vert a_j\vert^2<\eta$ (for example, take $a_j=1/J$ for large enough $J$). Then we may write $x$ as
$$x=\sum_{i=1}^I\sum_{j=1}^J a_j x_i=\sum_{i,j} v_{i,j}\, ,$$
where the vectors $v_{i,j}:=a_{j}x_{i}$ belong to $\mathcal E_{\rm cond}(T)$. 

\smallskip
By Fact \ref{factind}, one can find unimodular eigenvectors $u_{i,j}$ such that $u_{i,j}$ is very close to $v_{i,j}$ for all $(i,j)$, so that in particular $\sum_{i,j}u_{i,j}\in\Omega$, and the associated family of eigenvalues 
$(\lambda_{i,j})$ is independent. 

\smallskip
Now, consider the Steinhaus measure $\mu\sim \sum \chi_{i,j} u_{i,j}$. By Fact \ref{factomega}, the measure $\mu$ is supported on $\Omega$ (actually, its closed support is contained in $\Omega$). Moreover, $\int_X \Vert x\Vert^2 d\mu(x)=\mathbb E\left\Vert \sum_{i,j} \chi_{i,j} u_{i,j}\right\Vert^2$ can be made as close to $\mathbb E\left\Vert \sum_{i,j} \chi_{i,j} v_{i,j}\right\Vert^2$ as we wish, provided that the $u_{i,j}$ are close enough to the $v_{i,j}$ for all $(i,j)$. So it is enough to check that $\mathbb E\left\Vert \sum_{i,j} \chi_{i,j} v_{i,j}\right\Vert^2<\eta$.

\smallskip
We have 
\begin{eqnarray*}
\mathbb E \left\Vert \sum_{i,j} \chi_{i,j} v_{i,j}\right\Vert^2&=&\mathbb E\left\Vert \sum_{i=1}^I\left(\sum_{j=1}^J \chi_{i,j}a_j \right) x_i\right\Vert^2\\
&\leq & M\, \sum_{i=1}^I \mathbb E\left\vert \sum_{j=1}^J \chi_{i,j} a_j\right\vert^2\\
&=&  MI\, \sum_{j=1}^J a_j^2\, ,
\end{eqnarray*}
where at the third line we have used the fact that the Steinhaus variables $\chi_{i,j}$ are orthogo\-nal and normalized in $L^2$. Since $MI\, \sum_j a_j^2<\eta$, this concludes the proof.
\end{proof}

\begin{remark} The above proof shows that if $T$ admits a perfectly spanning set of unimodular eigenvectors, then it satisfies the assumptions of Theorem \ref{th8}. Hence, we have in fact a Baire category proof of the following result:  \emph{if $T$ admits a perfectly spanning set of unimodular eigenvectors, then it admits an ergodic measure $\mu$ with full support such that $\int_X\Vert x\Vert^2d\mu(x)<\infty$}. As mentioned above, this is equivalent to the existence of an ergodic Gaussian measure with full support if the Banach space $X$ has type $2$.
\end{remark}

\begin{remark} With minor adjustments, the same proof works for an operator acting on an arbitrary Polish topological vector space $X$. This may have some interest, since it was not quite clear from the existing proofs that the implication (perfect spanning property)$\implies$(ergodic measure with full support) holds true in this general setting.

Here are a few more words about the modifications needed in the proof. Start with a neighbourhood $W$ of $0$ in $X$, and for safeness  take another neighbourhood $W'$ of $0$ such that $W'+W'\subset W$. We are looking for a suitable Steinhaus measure ``almost" supported on $W$. 
Having found the vectors $x_1,\dots ,x_I$ as above, let $\Vert\,\cdot\,\Vert$ be any norm on the finite-dimensional space 
$E:={\rm span}(x_1,\dots ,x_I)$. Define $M$ and $a_1,\dots ,a_J$ as above. If the $u_{i,j}$ are chosen close enough to the $v_{i,j}$, the following implication holds for any 
$\mu_{i,j}\in \TT$: if $\sum_{i,j}\mu_{i,j}v_{i,j}\in W'$, then $\sum_{i;j}\mu_{i,j} u_{i,j}\in W$. So it is enough to ensure that $\sum_{i,j} \chi_{i,j} v_{i,j}$ belongs to $W'$ with large probability, which is done as above.
\end{remark}

\smallskip
The following result can be deduced from the proof of Theorem \ref{th6} (b).
\begin{proposition}\label{book} Let the Banach space $X$ have type $p\in (1,2]$. Assume that there exists some finite constant $C$ such that the following holds:
\par\smallskip
for any backward $T\,$-$\,$invariant open set $\Omega\neq\emptyset$, one can find a finite sequence of unimodular eigenvectors $(u_k)_{k\in K}$ whose associated eigenvalues $\lambda_k$ form an independent family, such that $\sum_k u_k\in \Omega$ and $\sum_k\Vert u_k\Vert^p\leq C$.
\par\smallskip
Then $\overline{\mathcal S_T}(HC(T))$ is a dense $G_\delta$ subset of 
$\overline{\mathcal S_T}(X)$.
\end{proposition}
\begin{proof} Since $X$ has type $p$, there exists a finite constant $C_p$ such that, for every finite family of vectors $(u_k)_{k\in K}\subset X$, we have
$$\mathbb E\left\Vert \sum_{k\in K} \chi_k u_k\right\Vert^p\leq C_p \sum_{k\in K} \Vert u_k\Vert^p .$$
By Fact \ref{factomega}, it follows that for any for any backward $T\,$-$\,$invariant open set $\Omega\neq\emptyset$, one can find a measure 
 $\mu\in\mathcal S_T(\Omega)$ such that $\int_X \Vert x\Vert^pd\mu(x)\leq C':=CC_p$. So we may apply Theorem \ref{th8}.
\end{proof}

\begin{remark} The assumption of Proposition \ref{book} is satisfied if there exists a sequence of unimodular eigenvectors 
$(x_k)_{k\in\NN}$ for the operator $T$ such that
\begin{itemize}
\item[\rm (i)] the eigenvalues $\lambda_k$ associated to the eigenvectors $x_{k}$ are independent;
\item[\rm (ii)] $\sum\limits_{n\in\NN} \Vert x_k\Vert^p<\infty$;
\item[\rm (iii)] $\sum\limits_{n\in\NN} \vert\langle x^*,x_k\rangle\vert=\infty$ for every nonzero linear functional $x^*\in X^*$.
\end{itemize}
 
 Indeed, assumption (iii) implies that the set 
$$\left\{ \sum_{k=1}^K a_k x_k;\; K\in\NN\;,\; \vert a_N\vert,\dots ,\vert a_K\vert\leq 1\right\}$$
is dense in $X$ (see \mbox{e.g.} \cite[Lemma 11.11]{BM1}). So, given a backward $T$-$\,$invariant nonempty open set $\Omega$,
one can find complex numbers $a_1,\dots , a_K$ with 
$0<\vert a_k\vert\leq 1$ such that $\sum_{k} a_kx_k\in \Omega$. Setting $u_k:=a_kx_k$, the $u_k$ are unimodular eigenvectors whose eigenvalues form an independent family,  such that $\sum_k u_k\in\Omega$ and $\sum_k\Vert u_k\Vert^p\leq C:=\sum_{1}^\infty \Vert x_k\Vert^p$.

\end{remark}

\subsection{Periodic measures and proof of Theorem \ref{th6} (a)} Recall that we denote by $\mathcal F_T(X)$ the convex hull of the set of all periodic measures for an operator $T\in\mathfrak L(X)$. Equivalently, $\mathcal F_T(X)$ is the family of all finitely supported $T\,$-$\,$invariant measures on $X$. 

\smallskip
Assume that $T$ satisfies the assumptions of Theorem \ref{th6} (a); that is, $T$ has a perfectly spanning set of unimodular eigenvectors and, moreover, any unimodular eigenvector can be approximated as close as we wish by a periodic eigenvector. Under these assumptions, we are going to show that the measure $m$ given by Theorem \ref{blackbox} belongs to ${\overline{\mathcal F}}_T(X)$. Once this is done, we may apply Theorem \ref{th7} with $\mathcal M=\mathcal F_T(X)$ to conclude that
${\overline{\mathcal F}}_T(HC(T))$ is a dense $G_{\delta }$ subset of ${\overline{\mathcal F}}_T(X)$. 
\par\smallskip

Let us say that a measure $\nu\in\mathcal P(X)$ is a \emph{periodic Steinhaus measure} for $T$ if $\nu$ is the distribution of a random variable $\Psi=\sum \chi_j u_j$, where the sum is finite and the $u_j$ are periodic unimodular eigenvectors for $T$. 

\begin{fact}\label{closure1} Any Steinhaus measure for $T$ lies in the closure of the periodic Steinhaus measures.
\end{fact}
\begin{proof} Let $\mu\sim \sum_{j\in J} \chi_j  x_j$ be an arbitrary Steinhaus measure for $T$. Our assumption on $T$ implies that for each $j\in J$, one can find a sequence of periodic eigenvectors $(u_{j,n})_{n\in\NN}$ for $T$ converging to $x_j$. Then the periodic Steinhaus measures $\nu_n\sim \sum_{j\in J} \chi_j u_{j,n}$ converge to $\mu$ 
in $\mathcal P(X)$.
\end{proof}

\begin{fact}\label{closure2} Any periodic Steinhaus measure for $T$ belongs to ${\overline{\mathcal F}}_T(X)$.
\end{fact}
\begin{proof} Let $\nu$ be a periodic Steinhaus measure for $T$, and write
$$\nu\sim \sum\limits_{j=1}^s \chi_j\, u_j\, ,$$
where the $u_j$ are periodic eigenvectors for $T$.

Let us fix an independent subset $\{ \lambda_1,\dots ,\lambda_s\}$ of $\TT$. 
 Let us also choose $Q\geq 1$ such that $T^Qu_j=u_j$ for all $j\in\{ 1,\dots ,s\}$. 

For each $N\geq 1$, consider the measure $\nu_N\in\mathcal P(X)$ defined by
$$\nu_N:=\frac{1}{N^s} \sum_{n_1,\dots ,n_s=0}^{N-1}\left( \frac1Q \sum_{j=0}^{Q-1} \delta_{T^j(\lambda_1^{n_1} u_1+\cdots  +\lambda_s^{n_s} u_s)}\right).$$
Since any point $a\in X$ of the form $a=\lambda_1^{n_1} u_1+\cdots  +\lambda_s^{n_s} u_s$ satisfies $T^Qa=a$, the measures $\nu_N$ belong to ${{\mathcal F}}_T(X)$.

If $f$ is a bounded continuous function on $X$ and $N\ge 1$, then
\begin{eqnarray*}
\int_X f\, d\nu_N&=&\frac1{N^s} \sum_{n_1,\dots ,n_s=0}^{N-1}\left(\frac1Q \sum_{j=0}^{Q-1} f\bigl( T^j(\lambda_1^{n_1} u_1+\cdots  +\lambda_s^{n_s} u_s)\bigr)\right)\\
&=&\frac1Q\sum_{j=0}^{Q-1} \left( \frac1{N^s} \sum_{n_1,\dots ,n_s=0}^{N-1} f\bigl(\lambda_1^{n_1} T^ju_1+\cdots  +\lambda_s^{n_s} T^ju_s\bigr)\right) .
\end{eqnarray*}
By the (multi-dimensional) Weyl equidistribution theorem, it follows that
$$\int_X f\, d\nu_N\To \frac1Q\sum_{j=0}^{Q-1} \int_{\TT^s}  f\bigl(\alpha_1 T^ju_1+\cdots  +\alpha_s T^ju_s\bigr)\, d\alpha_1\cdots d\alpha_s\quad 
\textrm{ as } {N\to\infty}.$$

Observe now that 
$$\int_X f\, d\nu=\mathbb E\left[ f\left(\sum_{j=1}^s \chi_j u_j\right)\right]=\int_{\TT^s} f(\alpha_1 u_1+\cdots +\alpha_s u_s)\, d\alpha_1\cdots d\alpha_s $$
and that, by the very same computation,
$$\int_X (f\circ T^j)\, d\nu=\int_{\TT^s}  f\bigl(\alpha_1 T^ju_1+\cdots  +\alpha_s T^ju_s\bigr)\, d\alpha_1\cdots d\alpha_s $$
for all $j\geq 0$. Since $\nu$ is $T$-$\,$invariant, it follows that 
$$\int_X f\, d\nu=\int_{\TT^s}  f\bigl(\alpha_1 T^ju_1+\cdots  +\alpha_s T^ju_s\bigr)\, d\alpha_1\cdots d\alpha_s$$
for all $j\in\{ 0,\dots ,Q-1\}$.

Altogether, we conclude that 
$$\int_X f\, d\nu_N\To \int_X f\, d\nu\quad \textrm{ for every }f\in \mathcal C_b(X),$$  \mbox{i.e.} that $\nu_N$ converges to $\nu$ in $\mathcal P(X)$ as 
$N$ tends to infinity. This shows that $\nu$ belongs to ${\overline{\mathcal F}}_T(X)$.
\end{proof}

\par\smallskip
By Facts \ref{closure1} and \ref{closure2}, we see that $\overline{\mathcal S}_T(X)\subset \overline{\mathcal F}_T(X)$. In particular, the measure $m$ given by Theorem \ref{blackbox} belongs to 
$\overline{\mathcal F}_T(X)$, which concludes the proof of Theorem 
\ref{th6}.
\qed
\par\smallskip
\begin{remark} The proof of Theorem \ref{th6} (a) consists in showing that any Steinhaus measure for $T$ belongs to $\overline{\mathcal F}_T(X)$. It would be interesting to know if this is holds true for every operator $T$ admitting a dense set of periodic points.
\end{remark}

\begin{remark} Theorem \ref{th6} (a) still holds, with the same proof, under the following formally weaker assumption: for every countable set $D\subset\TT$, the linear span of the spaces
$ \ker(T-\lambda )\cap \overline{\mathcal E_{{\rm per}}(T)}$, ${\lambda\in \TT\setminus D}$, is dense in $X$, where $\mathcal E_{{\rm per}} (T)$ is the set of all periodic eigenvectors of $T$.
This assumption is satisfied for all known chaotic operators in the literature. For example, it is extremely easy to
check if $T$ admits a 
``continuous and spanning field of unimodular eigenvectors" defined on a nontrivial arc of $\TT$; 
in other words, if there exists a a nontrivial arc $\Lambda\subset \TT$ and a continuous map $E:\Lambda\to X$ such that $TE(\lambda)=\lambda E(\lambda)$ for every 
$\lambda\in\Lambda$ and $\overline{\rm span}\, \{ E(\lambda);\; \lambda\in\Lambda\}=X$. This in particular 
to all operators satisfying the so-called \emph{Chaoticity Criterion} (which is the same as the Frequent Hypercyclicity Criterion, see \cite{BM1} or \cite{GEP}). Regarding this criterion, it is worth recalling here that if an operator $T$ satisfies it, then $T$ admits a mixing measure with full support (see \cite{MAP} or \cite{BM2}).
\end{remark}

\smallskip
It is quite plausible that in fact any chaotic operator with a perfectly spanning set of unimodular eigenvectors satisfies the additional assumption of Theorem \ref{th6} (a). 
We state this as
\begin{question}\label{Q11} Is it true that if $T$ is a chaotic operator with a perfectly spanning set of unimodular eigenvectors, then the periodic \emph{eigenvectors} of $T$ are dense in the set of all unimodular eigenvectors? Is it true at least that the conclusion of Theorem \ref{th6} (a) is valid for all chaotic operators with a perfectly spanning set of unimodular eigenvectors?
\end{question} 

\smallskip
We also note that Theorem \ref{th6} does not say that the \emph{ergodic} measures supported on $HC(T)$ are dense in ${\overline{\mathcal F}}_T(X)$. So one may ask

\begin{question}\label{Q9} Under the assumptions of Theorem \ref{th6}, is it true that the ergodic measures $m\in{\overline{\mathcal F}}_T(X)$ supported on $HC(T)$ form a dense $G_{\delta }$ subset of ${\overline{\mathcal F}}_T(X)$? Is it true at least that $T$ admits an ergodic measure $m$ with full support such that 
$m\in\overline{\mathcal F}_T(X)$?
\end{question}

\subsection{Another class of invariant measures}\label{blablaenplus} Let $T$ be a bounded operator on a complex Banach space $X$, and let us denote by $\mathcal S^*_T(X)$ the family of all probability measures $\mu$ on $X$ of the following form: 
$$\mu\sim\sum_{k\in K} \theta_k u_k\, ,$$
where $u_k$ is a unimodular eigenvector for $T$ whose associated eigenvalue $\lambda_k$ is a root of unity, $\theta_k$ is a random variable uniformly distributed on the (finite) subgroup 
$\Gamma_k$ of $\TT$ generated by $\lambda_k$, and the $\theta_k$ are independent. 

\smallskip
It is clear that any measure $\mu\in \mathcal S^*_T(X)$ is $T\,$-$\,$invariant, and that $S^*_T(X)$ is an adequate family of measures.
We are going to prove two analogues of Theorem \ref{th6} for this particular family.

\medskip
Let us denote by $\mathcal E^*(T)$ the set of all unimodular eigenvectors $v$ for $T$ satisfying the following property: for any neighbourhood $V$ of $v$ and any positive integer $d$, one can find a periodic unimodular eigenvector $u$ with associated eigenvalue $\lambda$ such that $u\in V$ and $\gcd({\rm o}(\lambda), d)=1$, where  ${\rm o}(\lambda)$ is the order of the root of unity $\lambda$. Our first result reads as follows.

\begin{proposition}\label{chinese} If the linear span of $\mathcal E^*(T)$ is dense in $X$, then $\overline{\mathcal S^*_T}(HC(T))$ is a dense $G_\delta$ subset of $\overline{\mathcal S^*_T}(X)$ and hence $T$ admits an ergodic measure with full support.
\end{proposition}

From this, we immediately deduce
\begin{corollary}\label{corochinese} Assume that the periodic vectors for $T$ are dense in $X$, and that for any periodic eigenvector $v$ and any positive integer $d$, one can find a periodic eigenvector $u$ arbitrarily close to $v$ whose associated eigenvalue $\lambda$ satisfies $\gcd({\rm o}(\lambda),d)=1$. Then $T$ admits an ergodic measure with full support.
\end{corollary}

Actually, it is known (see for instance Proposition \ref{psp} below) that $T$ even has a perfectly spanning set of unimodular eigenvectors if it satisfies the assumptions of Corollary \ref{corochinese}. On the other hand, it is not clear to us whether this is true or not when $T$ satisfies the \emph{a priori} weaker assumptions of Proposition \ref{chinese}; but one can prove in a very indirect way that this is indeed true when $T$ is a \emph{Hilbert space} operator (see Section \ref{finalsection}).

\begin{proof}[Proof of Proposition \ref{chinese}] This is very similar to the proof of Theorem \ref{th6} (b), so we shall be rather sketchy.

\begin{fact} Given $v_1,\dots ,v_N\in\mathcal E^*(T)$, one can find periodic eigenvectors $u_1,\dots ,u_N$ arbitraily close to $v_1,\dots ,v_N$ whose associated eigenvalues have pairwise coprime orders.
\end{fact}
\begin{proof}[Preuve] Start with a periodic eigenvector $u_1$ close to $v_1$ whose 
eigenvalue has order $d_1$; then, choose a periodic eigenvector $u_2$ close to $v_2$ whose
eigenvalue has order $d_2$ with $\gcd(d_2,d_1)=1$; then take a periodic eigenvector $u_3$ close to $v_3$ whose 
eigenvalue has order $d_3$ with $\gcd(d_3,d_1d_2)=1$, and so on.
\end{proof}

\begin{fact}\label{factcoprime} Let $(u_k)_{k\in K}$ be a finite family of periodic eigenvectors for $T$ whose associated eigenvalues $\lambda_k$ have pairwise coprime orders. Let also $\Omega\subset X$ be a backward $T\,$-$\,$invariant nonempty open set, and assume that $\sum u_k \in\Omega$. For each $k\in K$, denote by $\Gamma_k$ the subgroup of $\TT$ generated by $\lambda_k$. Then
$$\textrm{for every } (\mu_k)\in \prod_{k\in K} \Gamma_k , \textrm{ we have } \sum_{k\in K} \mu_k u_k\in  \Omega\, .$$
\end{fact}
\begin{proof}  By the Chinese remainder theorem, the subgroup of $\TT^K$ generated by $(\lambda_k)_{k\in K}$ is equal to $\prod_k \Gamma_k$. So, given $(\mu_k)\in\prod_k\Gamma_k$, one can find $n\in\NN$ such that $\lambda_k^n=\mu_k^{-1}$ for all $k\in K$. Then $T^n\left(\sum \mu_ku_k\right)=\sum u_k\in\Omega$, and hence $\sum\mu_k u_k\in\Omega$ because $\Omega$ is backward $T\,$-$\,$invariant.
\end{proof}

With the above two Facts at hand, one can now proceed exactly as in the proof of Theorem \ref{th6} (b), and this yields Proposition \ref{chinese}.
\end{proof}

\begin{remark}\label{encoreune} The above proof shows that if $T$ satisfies the assumption of Proposition \ref{chinese}, then one can find a measure $m\in\overline{\mathcal S^*_T}(HC(T))$ such that $\int_X\Vert x\Vert^2<\infty$: this follows from Theorem \ref{th8}. Hence, if $X$ has type $2$ and $\overline{\rm span}\,(\mathcal E^*(T))=X$, then $T$ admits an ergodic Gaussian measure with full support.
\end{remark}

\smallskip
Our second result is the ``periodic" analogue of Proposition \ref{book}.
\begin{proposition}\label{chinese+type} Let the Banach space $X$ have type $p\in (1,2]$. Assume that there exists some finite constant $C$ such that the following holds:
\par\smallskip
{for any backward $T\,$-$\,$invariant open set $\Omega\neq\emptyset$, one can find a finite sequence of unimodular eigenvectors $(u_k)_{k\in K}$ whose eigenvalues 
have pairwise coprime orders, such that $\sum_k u_k\in \Omega$ and $\sum_k \Vert u_k\Vert^p\leq C$.} 
\par\smallskip
Then $\overline{\mathcal S_T^*}(HC(T))$ is a dense $G_\delta$ subset of $\overline{\mathcal S_T^*}(X)$. This holds in particular if  there exists a sequence of periodic unimodular eigenvectors 
$(x_k)_{k\in\NN}$ whose associated eigenvalues have pairwise coprime orders, such that $\sum_{1}^\infty \Vert x_k\Vert^p<\infty$ and $\sum_{1}^\infty \vert\langle x^*,x_k\rangle\vert=\infty$ for every nonzero linear functional $x^*\in X^*$.
\end{proposition}
\begin{proof} One may proceed exactly as for the proof of Proposition \ref{book}, using Fact \ref{factcoprime} instead of Fact \ref{factomega}.
\end{proof}

\smallskip
In view of the above two results, the following question is of course quite natural, since a positive answer to it would imply that every chaotic operator admits an ergodic measure with full support.
\begin{question} Does every chaotic operator satisfy the assumptions of Proposition \ref{chinese}, and/or that Proposition \ref{chinese+type}?
\end{question}

\section{The perfect spanning property}\label{finalsection} In this final section, we collect  various equivalent formulations of the perfect spanning property. There are essentially no new results here: everything is proved in \cite{BG2},  \cite{G} and \cite{BM2}, except Corollaries \ref{coroindirect} and \ref{presque}. However, for the convenience of the reader it seems worth giving self-contained proofs of the ``elementary" equivalences.

\smallskip
In what follows, $T$ is a continuous linear operator on a complex Polish topological vector space $X$. 

\smallskip Recall that $T$ is said to have a \emph{perfectly spanning set of unimodular eigenvectors} if $\overline{\rm span}\left[ \ker(T-\lambda I);\; \lambda\in \TT\setminus D\right]=X$ for every countable set $D\subset \TT$.

\smallskip
We denote by $\mathcal E(T)$ the set of all unimodular eigenvectors for $T$. If $x\in\mathcal E(T)$, we denote the associated eigenvalue by $\lambda(x)$. The following three subsets of $\mathcal E(T)$ will be of interest for us: 
\begin{enumerate}
\item[$\bullet$] the set $\mathcal E_{\rm cond} (T)$ of all $v\in\mathcal E(T)$ such that, for any neighbourhood $V$ of $v$, there are uncountably many $\lambda\in\TT$ such that 
$\ker(T-\lambda I)\cap V\neq\emptyset$;
\item[$\bullet$]  the set $\mathcal E_{\rm acc} (T)$ of all $v\in\mathcal E(T)$ such that, for any neighbourhood $V$ of $v$, one can find $u\in\mathcal E(T)\cap V$ with $\lambda(u)\neq\lambda(v)$;
\item[$\bullet$] the set $\mathcal E^*(T)$ of all $v\in\mathcal E(T)$ such that, for every neighbourhood $V$ of $v$ and any positive integer $d$, one can find a periodic unimodular eigenvector 
$u$ in $V$ such that $\gcd({\rm o}(\lambda(u)),d)=1$.
\end{enumerate}

The notations $\mathcal E_{\rm cond}$ and $\mathcal E_{\rm acc}$ are meant to be reminiscent of the well known topological notions of \emph{condensation points} and  
\emph{accumulation points}. (We have no justification for the notation $\mathcal E^*$, except simplicity).

\smallskip
By a \emph{$\TT\,$-$\,$eigenvectorfield for $T$}, we mean a map $E:\Lambda\to X$ defined on some (nonempty) set $\Lambda\subset\TT$, such that $TE(\lambda)=\lambda E(\lambda)$ for every $\lambda\in\Lambda$.

\smallskip
Finally, given a probability measure $\sigma$ on $\TT$, we say that the set $\mathcal E(T)$ of unimodular eigenvectors is \emph{$\sigma$-spanning} if $\overline{\rm span}\left[ \ker(T-\lambda I);\; \lambda\in \TT\setminus D\right]=X$ for every Borel set $D\subset \TT$ such that $\sigma(D)=0$.

\begin{proposition}\label{psp} The following assertions are equivalent.
\begin{enumerate}
\item[\rm (1)] $T$ has a perfectly spanning set of unimodular eigenvectors;
\item[\rm (2)] The linear span of $\mathcal E_{\rm cond} (T)$ is dense in $X$;
\item[\rm (3)] There exists a set $\mathcal E\subset\mathcal E(T)$ such that $\overline{\rm span}\,(\mathcal E)=X$ and, for every $v\in \mathcal E$, one can find $u\in\mathcal E$ arbitrarily close to $v$ such that $\lambda(u)\neq\lambda(v)$;
\item[\rm (4)] one can find a sequence $(E_i)_{i\in\NN}$ of {continuous} $\TT\,$-$\,$eigenvectorfields  for $T$, where each $E_i$ is defined on some perfect set 
$\Lambda_i\subset\TT$, such that 
$\overline{\rm span}\left(\bigcup_{i\in \NN} E_i(\Lambda_i)\right)=X$;
\item[\rm (5)] $\mathcal E(T)$ is $\sigma\,$-$\,$spanning for some continuous probability measure $\sigma$ on $\TT$.
\end{enumerate}
Moreover, if the Banach space $X$ has \emph{cotype $2$}, these conditions are also equivalent to 
\begin{itemize}
\item[\rm (6)] $T$ admits an ergodic Gaussian measure with full support.
\end{itemize}
\end{proposition}
\begin{proof} $(1)\implies (2)$. Assume that $T$ has a perfectly spanning set of unimodular eigenvectors. Set 
$\mathbf V:=\{ (x,\lambda)\in X\times \TT;\; x\neq 0\;{\rm and}\; T(x)=\lambda x\}\,,$ and let us denote by $\mathbf O$ the union of all relatively open sets $O\subset \mathbf V$ such that 
$\pi_{\scriptscriptstyle \TT}(O)$ is countable, where $\pi_{\scriptscriptstyle \TT} :X\times\TT\to\TT$ is the canonical projection on the second coordinate. Then the set 
$D:=\pi_{\scriptscriptstyle \TT}(\mathbf O)$ is countable by the Lindel\"of property. So the linear span of $\mathcal E_{\TT\setminus D}:=\bigcup_{\lambda\in\TT\setminus D}\ker (T-\lambda)\setminus\{ 0\}$ 
is dense in $X$, by the perfect spanning property. Since $\mathcal E_{\TT\setminus D}$ is contained in $\mathcal E_{\rm cond} (T)$ by the definition of $D$, this proves (2).

\smallskip
$(2)\implies(3)$. Assume that $\overline{\rm span}\, (\mathcal E_{\rm cond} (T))=X$. We show that (3) is satisfied with $\mathcal E:= \mathcal E_{\rm cond} (T)$. Let us fix $v\in \mathcal E_{\rm cond} (T)$ and an open neighbourhood $V$ of $v$. Let us denote by $U_0$ the union of all open sets $U\subset V$ such that the set $\{ \lambda (u);\; u\in U\cap\mathcal E(T)\}$ is countable. Then the set $\{ \lambda(u);\; u\in U_0\cap\mathcal E(T)\}$ is countable by the Lindel\"of property. Since $v\in\mathcal E_{\rm cond} (T)$, it follows that $(V\setminus U_0)\cap\mathcal E(T)\neq\emptyset$. By the definition of $U_0$, any point $u\in(V\setminus U_0)\cap\mathcal E(T)$ belongs in fact to $\mathcal E_{\rm cond} (T)$; hence, (3) is indeed satisfied.

\smallskip
$(3)\implies (4)$. Assume that (3) is satisfied with some $\mathcal E\subset \mathcal E(T)$.

\begin{fact}\label{thelastone} For any $v\in\mathcal E$ and every neighbourhood $V$ of $v$, one can find a continuous $\TT\,$-$\,$eigenvectorfield $E:\Lambda\to X$ defined on some perfect set 
$\Lambda\subset\TT$ such that $E(\lambda)\in V$ for every $\lambda\in\Lambda$.
\end{fact}
\begin{proof} Fix some compatible complete metric on $X$. Set $v_\emptyset:=v$ and choose a neighbourhood $V_\emptyset$ of $v_\emptyset$ contained in 
$V$ such that ${\rm diam}(V_\emptyset)<2^{-0}$. Set also $\lambda_\emptyset:= \lambda(v_\emptyset)$, and choose an open neighbourhood $\Lambda_\emptyset$ of $\lambda_\emptyset$ in $\TT$ such that ${\rm diam}(\Lambda_\emptyset)<2^{-0}$. 

By (3), one can find a point $v_0\in V_\emptyset$ arbitrarily close to $v_\emptyset$ such that $\lambda(v_0)\neq\lambda(v_\emptyset)$. If we choose $v_0$ close enough to 
$v_\emptyset$, then $\lambda_0:=\lambda(v_0)\in \Lambda_\emptyset$. Next, since $v_0\in\mathcal E$, one can apply (3) again to find $v_1\in\mathcal E$ such that 
$\lambda(v_1)\neq \lambda(v_0)$ and close enough to $v_0$ to ensure that $v_1\in V_\emptyset$ and $\lambda_1:=\lambda(v_1)\in\Lambda_\emptyset$. Then, one may find neighbourhoods $V_0, V_1$ of $v_0, v_1$ with pairwise disjoint closures contained in $V_\emptyset$ and diameters less than $2^{-1}$, and neighbourhoods $\Lambda_0, \Lambda_1$ of $\lambda_0,\lambda_1$ with pairwise disjoint closures contained in $\Lambda_\emptyset$ and diameters less than $2^{-1}$.

Continue in the obvious way. In the end, this ``Cantor-like" construction produces a continuous and one-to-one map $\mathbf t\mapsto\lambda_{\mathbf t}$ from the Cantor space $\Delta:=\{ 0,1\}^\NN$ into $\TT$,  and a continuous map $F:\Delta\to X$ such that $F(\mathbf t)\in V$ for every $\mathbf t\in\Delta$ and $F(\mathbf t)$ is an eigenvector for $T$ with associated eigenvalue $\lambda_{\mathbf t}$. Then set $\Lambda:=\{ \mathbf\lambda_t;\; \mathbf t\in\Delta\}$ and $E(\lambda)=F(\mathbf t)$ for $\lambda=\lambda_{\bf t}\in\Lambda$.
\end{proof}

Now, let $(v_i)_{i\in\NN}$ be a sequence enumerating infinitely many times some countable dense subset of $\mathcal E$. By Fact \ref{thelastone}, one can find a sequence 
$(E_i)_{i\in\NN}$ of {continuous} $\TT\,$-$\,$eigenvectorfields  for $T$, where each $E_i$ is defined on some perfect set 
$\Lambda_i\subset\TT$, such that $\Vert E_i(\lambda)-v_i\Vert< 2^{-i}$ for each $i$ and every $\lambda\in\Lambda_i$. Then $\overline{\rm span}\left(\bigcup_{i\in\NN} E_i(\Lambda_i)\right)=X$ because $\overline{\rm span} \,(\mathcal E)=X$, and hence (4) is satisfied.

\smallskip
$(4)\implies (5)$. Let $(E_i)_{i\in\NN}$ be as in (4). For each $i\in\NN$, choose a continuous probability measure $\sigma_i$ on $\TT$ such that ${\rm supp}(\sigma_i)=\Lambda_i$; this is possible because $\Lambda_i$ is a perfect set. Set $\sigma:=\sum_{i\geq 1} 2^{-i} \sigma_i$. This is a continuous probability measure on $\TT$. If $D\subset\TT$ is a Borel set such that $\sigma(D)=0$, then $\sigma_i(D)=0$ for all $i$ and hence, $\Lambda_i\setminus D$ is dense in $\Lambda_i$ because the support of $\sigma_i$ is exactly equal to 
$\Lambda_i$. Since the $\TT\,$-$\,$eigenvectorfield $E_i$ is continuous, it follows that $\overline{E_i(\Lambda_i\setminus D) }=E_i(\Lambda_i)$ for each $i\in\NN$. So we have 
$\overline{\rm span}\,\bigcup_{i\in\NN} E_i(\Lambda_i\setminus D)=\overline{\rm span}\,\bigcup_{i\in\NN} E_i(\Lambda_i)=X$. Since 
$\bigcup_{i\in\NN} E_i(\Lambda_i\setminus D)\subset \bigcup_{\lambda\in\TT\setminus D} \ker(T-\lambda I)$, this shows that $\mathcal E(T)$ is $\sigma$-spanning.

\smallskip
$(5)\implies (1)$ is obvious since if $\sigma$ is a continuous measure on $\TT$, then $\sigma (D)=0$ for every countable set $D\subset\TT$.

\smallskip
Finally, $(1)\implies(6)$ is always true (without assumption on $X$) by \cite{BM2}, and it is shown in \cite{BG2} that the reverse implication holds true if $X$ has cotype $2$.
\end{proof}

\begin{corollary}\label{coroindirect} Assume that $X$ is a \emph{Hilbert space}. If the linear span of $\mathcal E^*(T)$ is dense in $X$, then $T$ has a perfectly spanning set of unimodular eigenvectors.
\end{corollary}
\begin{proof} If $\overline{\rm span}\,(\mathcal E^*(T))=X$, then $T$ admits an ergodic Gaussian measure with full support by Remark \ref{encoreune}, because $X$ has type $2$. Since $X$ also has cotype $2$, it follows that $T$ has a perfectly spanning set of unimodular eigenvectors.
\end{proof}

\begin{corollary}[\cite{G}] If $\overline{\rm span}\,(\mathcal E(T))=X$ and $\mathcal E_{\rm acc}(T)=\mathcal E(T)$, then $T$ has a perfectly spanning set of unimodular eigenvectors. In words: if the unimodular eigenvectors for $T$ span a dense subspace of $X$ and if every unimodular eigenvector $v$ can be approximated by unimodular eigenvectors $u$ with 
$\lambda(u)\neq\lambda (v)$, then $T$ has a perfectly spanning set of unimodular eigenvectors.
\end{corollary}
\begin{proof} Condition (3) is satisfied with $\mathcal E=\mathcal E(T)$.
\end{proof}

\begin{corollary}\label{presque} Assume that the periodic vectors for $T$ are dense in $X$, and that any periodic eigenvector $v$ for $T$ can be approximated by periodic eigenvectors $u$ with $\lambda(u)\neq \lambda(v)$. Then $T$ has a perfectly spanning set of unimodular eigenvectors.
\end{corollary}
\begin{proof} Take $\mathcal E$ to be the set of all periodic eigenvectors for $T$ in (3).
\end{proof}

At first sight, this last result does not seem to be very far from saying that any chaotic operator has a perfectly spanning set of unimodular eigenvectors. However, we have been unable to show that if $T\in\mathfrak L(X)$ is chaotic, then it satisfies the additional assumption in Corollary \ref{presque}. This leads to 
\begin{question} Assume that $T$ is a chaotic operator. Is it true that any periodic eigenvector $v$ for $T$ can be approximated by periodic eigenvectors $u$ such that $\lambda(u)\neq \lambda(v)$?
\end{question}

\end{document}